\documentclass{article}
\usepackage{graphicx}
\usepackage{amsmath} 
\usepackage{amssymb} 
\usepackage{amsthm}
\usepackage{mathrsfs}
\usepackage{tikz}
\usepackage{tikz-cd}
\usepackage[polutonikogreek,greek,english]{babel} 
\usepackage[top=2cm,bottom=2cm,left=3cm,right=3cm]{geometry}
\usepackage{stmaryrd}
\usepackage{hyperref}
\usepackage{dsfont}
\usepackage{bm}
\usepackage[T1]{fontenc}
\usepackage{mathabx}

\hypersetup{
    colorlinks=true,
    linkcolor=blue,
    filecolor=magenta,      
    urlcolor=cyan,
    pdfpagemode=FullScreen,
    }

\DeclareMathOperator{\rg}{\rm{rk}}
\newcommand{\abs}[1]{\left| #1 \right|} 
\DeclareMathOperator{\brau}{\rm{Br}}
\DeclareMathOperator{\brausub}{\brau_{\mathrm{sub}}}
\DeclareMathOperator{\brauvert}{\brau_{\mathrm{vert}}}
\newcommand{\pimp}[1]{#1_{\mathrm{odd}}}

\newcommand{\eN}{\mathbb{N}}
\newcommand{\eZ}{\mathbb{Z}}
\newcommand{\eQ}{\mathbb{Q}}
\newcommand{\eR}{\mathbb{R}}
\newcommand{\eC}{\mathbb{C}}

\newcommand{\eA}{\mathbb{A}}
\newcommand{\p}{\mathbb{P}}
\newcommand{\cN}{\mathcal{N}}
\newcommand{\cE}{\mathcal{E}}
\newcommand{\cM}{\mathcal{M}}

\newcommand{\cP}{\mathcal{P}}

\newcommand{\Mod}[1]{\ (\mathrm{mod}\ #1)}

\newtheorem{tho}{Theorem}[section]
\newtheorem{cor}[tho]{Corollary}
\newtheorem{lemme}[tho]{Lemma}
\newtheorem{defin}[tho]{Definition}
\newtheorem{prop}[tho]{Proposition}
\newtheorem{conj}[tho]{Conjecture}
\newtheorem{rmk}[tho]{Remark}

\title{Solubility of a family of conics with polynomial coefficients in many variables}
\author{ Mathieu Da Silva }

\begin{document}

\maketitle

\textbf{Abstract :} We study the proportion of conics given by $(\mathcal{C}_{\bm F, \bm y}) : F_0(\bm y)x_0^2 + F_1(\bm y)x_1^2 = F_2( \bm y)x_2^2 $ which have a rational point $\bm x = (x_0 :x_1:x_2) \in \p^2(\eQ)$, where $\bm y = (y_0 : \dots : y_n)\in \p^n(\eQ)$ and  $F_0,F_1,F_2 \in \eZ[X_0,\ldots, X_n]$ are homogeneous polynomials in many variables of the same degree $d$. We provide an asymptotic formula for the number of $\bm y$ of bounded height such that the corresponding conic $(\mathcal{C}_{\bm F, \bm y})$ has a rational point. In particular, our result agrees with the Loughran--Smeets and the Loughran--Rome--Sofos conjectures. Our strategy is based on a recent result of Destagnol--Lyczak--Sofos relying on the circle method to estimate the average of an arithmetic function over polynomials in many variables. To this end, we study the proportion of conics $t_0x_0^2 + t_1x_1^2 + t_2x_2^2 = 0$ having a rational point, and coefficients $t_0,t_1,t_2$ in arithmetic progressions.

\tableofcontents

\section*{Notation}
\begin{itemize}

 \item[$\bullet$] The letter $p$ is exclusively used for prime numbers. We write $\mathcal{P}$ for the set of primes. If~$n,m \in~\eN$, the notation $p^m \mid\mid n$ means that $p^m \mid n$ and $p^{m+1} \nmid n$. If $2^m \mid \mid n$, we denote by $\pimp{n}$ the integer $2^{-m} n$. 

    \item[$\bullet$] If ${\bm n} = (n_1,\ldots, n_k) \in \eZ^k$, we write $\abs{\bm n} := n_1 + \cdots + n_k$ and $\pimp{\bm n} := (\pimp{(n_1)}, \dots, \pimp{(n_k)})$.

    \item[$\bullet$] For $k \geqslant 2$, the notation $\log_k$ stands for the $k^{\text{th}}$ iterate of $\log $.
    
    \item[$\bullet$] If $f : \eN^k \to \eC$ is an arithmetic function, we write $f({\bm n})$ for $f(n_1, \ldots, n_k)$ and we denote its associated Dirichlet series by $D_f({\bm s}) = D_f(s_1,\ldots, s_k)$. We call $f$ multiplicative if it satisfies the equality $f(a_1 b_1, \dots, a_kb_k) = f(\bm a) f(\bm b)$ for all $\bm a, \bm b \in \eN^k$ such that $\gcd( a_1 \cdots a_k, b_1 \cdots b_k) = 1$.

    \item[$\bullet$] If $s$ is a complex number, we denote by $\sigma $ its real part.
    
    \item[$\bullet$] For $z \in \eC$ we denote by $\tau_z(n)$ the $n$-th coefficient in the Dirichlet series expansion of $\zeta(s)^z$ ($\sigma) > 1$). Its existence is justified for instance in \cite[chap. II. §5.1]{Tenenbaum}. If ${\bm z} = (z_1,\ldots, z_k) \in \eC^k$, we~introduce~$\boldsymbol{\tau}_{\bm z}({\bm n}):= \prod_{i = 1}^k \tau_{z_i}(n_i) $. The latter defines an arithmetic function in $k$~variables~which~is~multiplicative~and~generalises~the~$\tau_z$~function.\\

    \item[$\bullet$] The Gamma function is defined by $\Gamma(z) := \int_0^{+\infty} e^{-t} t^{z-1} \mathrm{d}t$ for $\Re(z) > 0$. We recall that it admits a meromorphic continuation on $\eC$ with simple poles at each non-positive integer.
    
    \item[$\bullet$] If $\ell_0,\ell_1,\ell_2$ (\textit{resp.} $m_{01}, m_{02},m_{12}$) are three integers, we set $\ell_{012} := \ell_0\ell_1\ell_2$ (\textit{resp.} $m_{012} := m_{01}m_{02}m_{12}$).\\
    
    \item[$\bullet$]  The notation $\eA_\eQ$ stands for the ring of adèles of $\eQ$, that is the set of sequences $(a_\infty, a_2, a_3 , \dots, a_p, \dots)$ where $a_\infty \in \eR$ and $a_p \in \eQ_p$ for all $p$ with $a_p \in \eZ_p$ for all but finitely many $p$.\\
    
    \item[$\bullet$] If $K$ is a field and if $\bm t = (t_0,t_1,t_2) \in K^3$, we set 
    \begin{equation*}\label{notation} \tag{0.1}\vartheta_K( \bm t) := \mathds{1}\left(t_0y_0^2 + t_1y_1^2 = t_2y_2^2 \text{ has a solution $\bm y \in K^3 \smallsetminus \{(0,0,0)\}$}\right).\end{equation*} If $p$ is a prime and $\bm t \in \eQ^3$, we simply write $\vartheta_p(\bm t)$ for $\vartheta_{\eQ_p}(\bm t)$. We denote by $(\cdot ,\cdot )_K$ the Hilbert symbol over $K$ (see \cite[chap. 3]{Serre2}). If $K = \eQ_p$, we prefer the notation $(\cdot , \cdot)_p$ instead of $(\cdot, \cdot )_{\eQ_p}$.
    \item[$\bullet$] In this paper, an algebraic variety is a $\eQ$-scheme which is reduced, separated and of finite type. If $X$ is an algebraic variety, we denote by $K_X$ its canonical divisor and by $\Lambda_{\mathrm{eff}}(X)$ its pseudo-effective cone, \textit{i.e.} the closure of the cone of effective divisors in $\mathrm{Pic}(X) \otimes_\eZ \eR$. For each point $x$ of a scheme $X$, we denote by $\kappa(x)$ its residue field. The set of codimension one points of $X$ is denoted $X^{(1)}$.
    \item[$\bullet$] Finally, we denote by $\mu_{\infty}$ the Lebesgue measure on $\eR^{n+1}$ and by $\mu_p$ the Haar measure on $\eQ_p^{n+1}$, normalised by $\mu_p(\eZ_p^{n+1} ) = 1$. 
\end{itemize}

\section{Introduction}\label{sec III1}

\paragraph{Motivation and first definitions.}
 Let $n \in \eN$ and $\bm x  \in \p^n(\eQ)$. We recall that the anticanonical height on~$\p^n(\eQ)$ is defined by 
\[ H({\bm x}) := \max(\abs{x_0}, \ldots, \abs{x_n})^{n+1},\] where $[x_0 : \cdots : x_n]$ is a representative of $\bm x$ with coprime coordinates. The integer $H(\bm x)$ does not depend on the choice of such a representative of $\bm x$.\\ 
   
Let $Y$ be a smooth and proper projective variety over $\eQ$ and $\pi~:~Y~\longrightarrow~\p^n$ a dominant morphism whose generic fibre is geometrically integral. The fibres of $\pi$ define a family of varieties. The problem we are interested in consists in estimating, when $B$ goes to $+ \infty$, the quantity  
\[ N_{\mathrm{loc}}(\pi, B) :=\# \left\{ \bm x \in \p^n(\eQ) : H(\bm x) \leqslant B, \bm x \in \pi( Y(\eA_\eQ))  \right\},\]which is a natural upper bound for the quantity
\[ N_{\text{glob}}(\pi,B) := \# \left\{ \bm x \in \p^n(\eQ) : H(\bm x) \leqslant B, \bm x \in \pi (Y(\eQ)) \right\}.\]

\paragraph{The Loughran--Smeets conjecture.}   In 2016, Loughran and Smeets \cite[th.~1.2 and 1.3]{LS} used the large sieve and the Ekedahl sieve to obtain an upper bound for $ N_{\mathrm{loc}}(\pi, B) $ (and hence for $N_{\mathrm{glob}}(\pi, B)$), which is valid for every dominant morphism $\pi : X \to \p^n$ whose generic fibre is geometrically integral, where $X$ is a smooth and proper projective variety over a number field. The Loughran--Smeets conjecture (see \cite[conj. 1.6]{LS}) states that this upper bound is sharp for $N_{\mathrm{loc}}(\pi, B)$ under some mild assumptions. More precisely, it predicts that 
\[ N_{\mathrm{loc}}(\pi, B) \underset{B \to + \infty}{\sim} c {B \over (\log B)^{\Delta(\pi)}} \]for some constant $c > 0 $ and $\Delta(\pi) \in \eQ^+$ defined by \cite[(1.3)]{LS}. This generalises a previous work of Serre \cite{Serre1} from 1990 where, using the large sieve as well, the author obtained a sharp upper bound for the quantity 

\[ N(B) := \# \left\{ {\bm t} \in \left[ -B,B \right]^3 : t_0 y_0^2 + t_1 y_1^2 = t_2y_2^2 \text{ has a solution $\bm y \in \eQ^3 \smallsetminus \{(0,0,0)\}$}\right\} \quad (B \geqslant 2).\]Note that here, studying the quantity $N(B)$ is equivalent to studying the quantity $N_{\mathrm{loc}}(\pi, B)$ in the case of the family described by   
\[ X : \quad  t_0 y_0^2 + t_1y_1^2 = t_2y_2^2 \quad \subseteq \p^2 \times \p^2, \]with $\pi : (\bm y, \bm t) \in Y \mapsto \bm t \in \p_\eQ^2$. Between 1990 and 2000 matching lower bounds for $N(B)$ were obtained independently by Guo \cite{Guo} and Hooley \cite{Hooley1,Hooley2}. In 2023, Loughran, Rome and Sofos promoted those bounds to an asymptotic by showing the estimate
\[\label{question serre} \tag{1.1} N(B) \underset{B\to + \infty}{\sim} c {B^3 \over (\log B)^{3/2}} \]~where $c$ is an explicit positive constant. They also gave a prediction for the shape of the constant $c$ in general.\\

We now recall some terminology in order to state conjecture \cite[3.8]{LRS}. For any weak Fano variety~$X$ (see \cite[def. 3.1]{LRS}), we denote by $\Lambda_{\mathrm{eff}}(X)^{\vee} \subseteq \mathrm{Pic}(X)^{\vee} $ the dual of the pseudo-effective cone, equipped with the Haar measure such that the dual lattice $\mathrm{Pic}(X)^{\vee}$ has covolume $1$. The purity exact sequence \cite[3.7.2]{CTS} yields
    \[ 0 \longrightarrow  \brau(X) \longrightarrow  \brau( \eQ(X)) \overset{\oplus \partial_D}{\longrightarrow } \bigoplus_{D \in X^{(1)}} H^1 (\eQ(D), \eQ / \eZ)\]where $\partial_D : \brau( \eQ(X)) \longrightarrow  H^1 (\eQ(D), \eQ / \eZ)$ is the residue map at $D$. If $b$ is such that $\partial_D(b) = 0$, one says that $b$ is unramified along $D$. Let $Y$ be a smooth and proper projective variety over a number field and $\pi : Y \longrightarrow X$ be a dominant morphism. The subordinate Brauer $\brausub (X,\pi)$ is defined by 
    \[ \brausub (X,\pi) := \bigcap_{D \in X^{(1)}} \left\{ b \in \brau \eQ(X) : \begin{tabular}{c} $\partial_E \; \pi^\ast b = 0 \text{ for all irreducible components }$ \\
    $E \subseteq \pi^{-1} (D) \text{ of multiplicity }1 $ \end{tabular}\right\}.\]Note that these two notions come from \cite{LRS}. We also need the following definitions.

\begin{defin}
  Let $K$ be a number field or a finite field, let $\overline{K}$ be its algebraic closure, and let $X$ be a finite type
scheme over $K$. The absolute Galois group $\mathrm{Gal}(\overline{K}/K)$ acts on the geometric irreducible
components of $X$, i.e. the irreducible components of $X \times_{K} \overline{K}$. We say that $X$ is split (resp. pseudo-split) if $\mathrm{Gal}(\overline{K}/ K)$ (resp. every element of $\mathrm{Gal}(\overline{K}/ K)$) fixes some geometric irreducible component of multiplicity 1.
\end{defin} 
\begin{defin}
 Let $\pi : Y \longrightarrow X$ be a dominant proper morphism of smooth varieties over $\eQ$. For each point $D$ of $X$, we denote by $\kappa(D)$ its residue field and we choose some finite group $\Gamma_D$ through which the absolute Galois group $\mathrm{Gal}( \overline{\kappa(D)}/ \kappa(D))$ acts on the geometric irreducible components of $\pi^{-1}(D)$. We define
\[ \delta_D(\pi) = { \# \{ \gamma \in \Gamma_D : \gamma \text{ fixes a geometric irreducible component
of } \pi^{-1}(D) \text{ of multiplicity 1} \} \over \#\Gamma_D}.\]
\end{defin}This quantity measures how "non-split" the fibre over $D$ is. In particular, if $\pi^{-1}(D)$ is split, then we have $\delta_D(\pi) = 1$. The converse does not hold in general as shown in \cite[rmk. 3.5]{DS}. However, we have that $\delta_D(\pi) = 1$ if and only if $\pi^{-1}(D)$ is pseudo-split. We are now ready to introduce   \begin{itemize}
    \item the normalised constant of the effective cone defined by $\theta(X) := \displaystyle \int_{v \in \Lambda_{\mathrm{eff}}(X)^{\vee}} \mathrm{e}^{- \langle -K_X,v \rangle } \mathrm{d}v$;
    \item the modified Fujita invariants $\eta(D):= \sup\{t >0 : -K_X - tD \in \Lambda_{\text{eff}}\}$, satisfying $\eta(D_i) = a_i$ if
    \[ -K_X = a_1D_1 + \cdots + a_rD_r,\]and where the $D_i$ are the generators of the effective cone of $X$;
     \item the modified Tamagawa measure denoted by $\tau_\pi \left( \left( \prod_p \pi ( U(\eQ_p ))\right)^{\brausub(\p^n, \pi)} \right) $, defined in~\cite[(3.5)]{LRS};
    \item the real number $\Gamma(X, \pi) = \prod_{D \in X^{(1)}} \Gamma(\delta_D(\pi))$;
    \item the invariant associated with the fibration $\pi$ defined by \cite[(1.3)]{LS} and equal to
\[\Delta( \pi) = \sum_{D \in X^{(1)}} (1 - \delta_D(\pi))\]by \cite[prop. 3.10]{LS}.
\end{itemize}

 In the case where the field is~$\eQ$ and where the base of the fibration $\pi$ is $\p^n$, the Loughran--Rome--Sofos conjecture can be written as follows.
 
 \begin{conj}\label{conj lrs} \cite[conj 3.8]{LRS}
 Let $Y$ be a smooth projective variety and $\pi : Y \longrightarrow \p^n$ a dominant morphism whose generic fibre is geometrically integral. Let $U_0$ be the open subset of~$\p^n$ obtained by removing the closure of all the closed points whose fibres by $\pi$ are non-split and let~$U:=~\pi^{-1}(U_0)$. Assume that \begin{enumerate}
     \item[$(i)$] the morphism $\pi$ has a smooth fibre above a rational point, which is everywhere locally soluble;
     \item[$(ii)$] the fibres above the codimension one points have an irreducible component of multiplicity one. 
 \end{enumerate} Then there exists a constant $c_{\mathrm{LRS}}(\pi) > 0$ such that, when $B$ goes to $+\infty$, we have 

\[ N_{\mathrm{loc}}(\pi, B) \sim c_{\mathrm{LRS}}(\pi) {B \over (\log B)^{\Delta(\pi)}} .\]The constant $c_{\mathrm{LRS}}(\pi)$ can be written using geometric invariants of $\p^n$ and $\pi$ as follows
\[ \label{cst conj} \tag{1.2}c_{\mathrm{LRS}}(\pi) = {\theta(\p^n) \cdot \# ( \brausub (\p^n, \pi) / \brau \eQ ) \cdot \tau_\pi \left( \left( \prod_p \pi ( U(\eQ_p ))\right)^{\brausub(\p^n, \pi)} \right) \over \Gamma(\p^n , \pi)} \prod_{D \in (\p^n)^{(1)}} \eta(D)^{1- \delta_D(\pi)}. \]
\end{conj}

\begin{rmk}
    The open subset $U$ is necessary to have a well-defined modified Tamagawa measure. In some applications (see for instance \cite[lemma 4.7]{LRS}), the open $U$ does not play any role, since its complement has measure $0$. We refer to Remark \ref{rmq ouverts} for more details on the role of the open sets $U$ and $U_0$ in the study of the family considered in this paper.
\end{rmk}

\begin{rmk}
   By choosing $\pi$ to be the identity map $X \longrightarrow X$, we obtain that the Loughran--Smeets conjecture \cite[3.8]{LRS} generalises the Manin--Peyre conjecture \cite[2.2.1 and 2.2.2]{Peyre}.
\end{rmk}

A list with the cases already studied is given in \cite{LRS}. We extend this list to include more recent results in the following table. \\

\begin{tabular}{|c|c|c|}
    \hline Nature & Equation defining the family & Author(s) \\
     \hline
     Conic bundle & $y_0^2 - D y_1^2 = F_0(\bm t)\cdots F_R(\bm t) $ where the $F_i$ have& Destagnol--Lyczak--Sofos  \\  over $\p^n $ & the same degree $d$, form a Birch system, & (2025) \cite{DLS}   \\
    & $D \neq 0$ is square-free and $d R$ is even& \\
     \hline 
     Fibration over $\p^2$ & Fermat equations $x_0y_0^n + x_1y_1^n + x_2y_2^n =0$ & Koymans--Paterson\\
  &  & --Santens--Shute (2025) \cite{KPSS}  \\
     \hline
    Family of quadrics & $x_0 y_0^2 + x_1 y_1^2+x_2y_2^2 + x_3 y_3^2 = 0$ & Wilson (2024) \cite{Wilson1}  \\ parametrised by  & with $x_0x_1 = x_2x_3$ &   \\
  $\p^1 \times \p^1$ &  & \\
     \hline  
\end{tabular}
\\
\\
 Note that in Wilson's recent work \cite{Wilson1}, the base does not satisfy assumption \cite[(1.3)]{LRS} and the order of magnitude has an extra factor $\log_2 B$. This motivates the study of $N_{\text{loc}}(\pi,B)$ under assumptions that differ from those of \cite[conj. 3.8]{LRS}.\\

 Keeping in mind that we aim to apply the circle method, we now recall the definition of a Birch system. 

\begin{defin}(Birch system)\label{polynomes de birch}
    Let $n \geqslant  0$, and $R,d \geqslant 1$ be integers. Let $F_0,\ldots, F_R \in~\eZ[X_0,\ldots, X_n]$ be homogeneous polynomials of the same degree $d$. Let $\sigma(\bm F)$ be the dimension of the complex subvariety of~$\eA_\eC^{n+1}$ defined by 
    \[ \rg \left( {\partial F_i \over \partial t_j} (\bm t)\right)_{\substack{ 0 \leqslant i \leqslant R \\ 0 \leqslant j \leqslant n}}  < R+1. \]We say that $\bm F = (F_0, \dots, F_R)$ forms a Birch system if we have  \[ {n +1  - \sigma(\bm F) \over 2^{d-1}} > (R+1)(R+2)(d-1).\]
\end{defin}

\paragraph{Main result.} Let $F_0,F_1,F_2 \in \eZ[X_0,\ldots, X_n]$ be polynomials forming a Birch system. For $\bm y \in \p^2(\eQ)$ and for~$\bm x \in~\p^n(\eQ)$, the equation  
    \[ F_0(\bm x) y_0^2 + F_1(\bm x) y_1^2 = F_2(\bm x) y_2^2 \]defines a subvariety $Y$ of $V := \p^2 \times \p^n$. Denote by $\pi_{\bm F} : Y \longrightarrow \p^n$ the projection $(\bm y, \bm x) \mapsto \bm x$. In this paper, we establish Conjecture \ref{conj lrs} for the fibration $\pi_{\bm F}$. Let us note that the Hasse--Minkowski theorem implies the equality
    \[ N_{\mathrm{loc}} (\pi_{\bm F},B)  = \# \left\{ \bm x \in \p^n(\eQ) : H(\bm x) \leqslant B, \; \vartheta_\eQ(F_0(\bm x), F_1(\bm x) , F_2(\bm x)) = 1\right\} =: N_{\text{glob}}(\pi_{\bm F },B), \]where the notation $\vartheta_\eQ$ is introduced in (\ref{notation}).

\begin{tho}\label{THA}
    Let $d \geqslant 1$ and $F_0,F_1,F_2 \in \eZ[X_0, \ldots, X_n]$ be three homogeneous polynomials of degree~$d$, forming a Birch system as defined in Definition \ref{polynomes de birch}. Assume that for all $i \in \{0,1,2\}$, the variety cut out by $F_i = 0$ is smooth and that the variety cut out by $(F_0 = F_1 = F_2= 0)$ is a complete intersection. When~$B$ goes to $+ \infty$, we have
    \[ N_{\mathrm{glob}} (\pi_{\bm F},B) \sim c_{\mathrm{LRS}}(\pi_{\bm F} ){B \over (\log B)^{3/2}} \]where $c_{\mathrm{LRS}}(\pi_{\bm F})$ is the constant (\ref{cst conj}) predicted by \cite[conj. 3.8]{LRS}.
\end{tho}

This result is established in subsection \ref{sec III6}. It follows from Theorem \ref{problème de comptage} combined with \cite[th. 2.4]{DLS}, which itself relies on the circle method.\\

\noindent \textbf{Outline of the proof : } In section \ref{sec III2}, we compute the subordinate Brauer group (see \cite[def.~2.1]{LRS}) of the fibration $\pi_{\bm F}$. This enables us to determine the constant $c_{\mathrm{LRS}}(\pi_{\bm F})$ in Theorem \ref{THA} that is conjectured by Loughran--Rome--Sofos. In section~\ref{sec III3} we show a general lemma, namely Lemma~\ref{lemme TN}, aiming to simplify the proof of Theorem \ref{problème de comptage}. This lemma follows from an application of the Selberg--Delange method \cite[chap. II. §5.4]{Tenenbaum}. In sections \ref{sec III4} and \ref{sec III5} we estimate the proportion of projective conics which have a rational point and whose coefficients lie in arithmetic progression (Theorem \ref{problème de comptage}). Finally, in section~\ref{sec III6} we use a result from Destagnol--Lyczak--Sofos relying on the circle method, which we combine with Theorem \ref{problème de comptage} to conclude the proof of Theorem~\ref{THA}.

\section{Prediction for the constant}\label{sec III2}

Let $\bm F = (F_0,F_1,F_2) \in \eZ[X_0, \dots, X_n]^3$ satisfying the assumptions of Theorem \ref{THA}. For $\bm y \in \p^2(\eQ)$ and $\bm x \in \p^n(\eQ)$, the equation  
    \[ F_0(\bm x) y_0^2 + F_1(\bm x) y_1^2 = F_2(\bm x) y_2^2 \]defines a subvariety $Y$ of $V := \p^2 \times \p^n$. Recall that $\pi_{\bm F} : Y \longrightarrow \p^n$ denotes the projection map~$(\bm y, \bm x)~\mapsto~\bm x$.

\paragraph{The subordinate Brauer group of $\pi_{\bm F}$.} Following \cite[§2.]{LRS}, we compute the number of elements of~$\brausub(\p^n , \pi_{\bm F}) / \brau \eQ$ where
    \[ \brausub (\p^n,\pi_{\bm F}) := \bigcap_{D \in (\p^n)^{(1)}} \left\{ b \in \brau \eQ(\p^n) : \begin{tabular}{c} $\partial_E \; \pi_{\bm F}^\ast b = 0 \text{ for all irreducible components }$ \\
    $E \subseteq \pi_{\bm F}^{-1} (D) \text{ of multiplicity }1 $ \end{tabular}\right\}.\]

 \begin{prop}
        The variety $Y$ is smooth.
        \end{prop}
        \begin{proof}
            It suffices to use an open covering of $Y$. For integers $k \in \{1, \dots, n\}$ and $i \in \{0, \dots, k\}$, let~$D_k(i)$ be the set of points $\bm x$ in $\p^k(\eQ)$ satisfying $x_i \neq 0$. We choose the covering defined by the open sets~$Y\cap(D_2(i)\times D_n(j))$ for $i \in \{0,1,2\}$ and $j \in \{ 0, \dots, n\}$. Then, $D_2(i)\times D_n(j)$ is isomorphic as a variety to the affine space of dimension $n+2$ over~$\eQ$. Without loss of generality, we can assume that $i = j = 0$. Let then $P = ((z_1,z_2),(t_1,\ldots, t_n))\in D_2(0) \times D_n(0) $, with $z_i = y_i/y_0$ and $t_i = x_i/x_0$, be such that \[ F_0(1,t_1, \ldots, t_n) + F_1(1,t_1,\ldots, t_n) z_1^2 - F_2(1,t_1,\dots, t_n) z_2^2 = 0.\]We now prove that the gradient at $P$ is non-zero. Suppose for contradiction that the gradient vanishes at $P$, and write the gradient $(J_{1, \ell}(P))_{0 \leqslant \ell \leqslant n+1}$ with
            \[ J_{1,0}(P) = 2F_1(1,t_1,\ldots, t_n) z_1, \; J_{1,1}(P) = -2F_2(1,t_1,\dots, t_n) z_2,   \]and 
            \[ \forall \ell \in \{2, \dots, n+1\}, \; J_{1, \ell} (P) =  \displaystyle {\partial F_0 \over \partial t_{\ell-1}}(1,t_1, \ldots, t_n) +  {\partial F_1 \over \partial t_{\ell-1}}(1,t_1,\ldots, t_n) z_1^2 - {\partial F_2 \over \partial t_{\ell-1}}(1,t_1,\ldots, t_n) z_2^2.\]Either $z_1 = z_2 = 0$, which contradicts the smoothness of the variety cut out by $(F_0 = 0)$, or $z_1$ or $z_2$ is non-zero and we deduce that at least two of the gradients of the $F_i$ at $(1, t_1, \dots, t_n)$ are linearly dependent, contradicting the fact that the variety cut out by $(F_0 = F_1 = F_2 = 0)$ is a complete intersection.
        \end{proof}

The main result of this section is the following. 

\begin{prop}\label{brauer subordonné}
    Let $n \geqslant 1$ and $\bm F = (F_0,F_1,F_2)$ be a Birch system of polynomials in $n+1$ variables, as in Definition \ref{polynomes de birch}. Let $\alpha \in \brau \eQ(\p^n)$ be the class of the quaternion algebra $(- F_1/F_0, F_2/F_0)$. We have a group isomorphism $\brausub (\p^n , \pi_{\bm F}) / \brau \eQ \simeq \eZ/ 2 \eZ$ and a generator is given by the class of $\alpha$.
\end{prop}
\begin{proof}
    Firstly, we assume that the morphism $\pi_{\bm F}$ is flat. Since $\bm F$ is a Birch system, we can assume~$n \geqslant 3$. Denote by $Z$ the common zero locus of $F_0,F_1$ and $F_2$. Since the variety cut out by $(F_0 = F_1 = F_2 = 0)$ is a complete intersection, $Z$ is a closed subset of $\p^n$ of codimension~$3$. Let $G := \pi_{\bm F}^{-1}(Z)$. It is a closed subset of $Y$ of codimension $3$ in the hypersurface $Y$ of $\p^2 \times \p^n$. Let $Y_1 := Y \smallsetminus G$ and $X := \p^n \smallsetminus Z$. Then, \cite[3.7.5]{CTS} implies the equality $\brau Y_1 = \brau Y$ and $\brau X \simeq \brau \p^n$ so that we have $  \brauvert(Y / \p^n) \simeq  \brauvert(Y_1 / X)$. Moreover, the projection $\pi_{\bm F}$ induces a morphism $\widetilde{\pi_{\bm F}} : Y_1 \longrightarrow  X$, which is flat by the miracle flatness lemma, and satisfying $ \brau_{\mathrm{sub}} (Y_1,\widetilde{\pi_{\bm F}})  \simeq \brau_{\mathrm{sub}} (Y,\pi_{\bm F}) $. Hence, \cite[lemma 2.3]{LRS} gives the equality
    \[ \brauvert(Y/ \p^n) = \pi_{\bm F}^\ast \brausub(\pi_{\bm F}, \p^n).\]We now compute $\brauvert(Y/\p^n)$. By \cite[2.6]{Schindler} we already know that the natural morphism $\brau~\eQ~\longrightarrow~\brau~Y$ is an isomorphism. Recall that we also have a short exact sequence \cite[(2.12)]{Loughran}, which can be written here as \[ \label{shortexactsequence}\tag{2.1}  0 \longrightarrow \langle \alpha \rangle \longrightarrow \brau \eQ(\p^n)  \overset{\pi_{\bm F}^\ast}{\longrightarrow}  \brau Y_{\eta'} \longrightarrow 0 \]~where $\eta'$ denotes the generic point of $Y$. In particular, we have $\pi_{\bm F}^\ast (\brau  \eQ(\p^n)  ) = \brau Y_{\eta'}$, and since $\brau Y $ is a subgroup of $\brau Y_{\eta'}$, it follows that
    \[ \brauvert(Y / \p^n) = \brau Y \cap \pi_{\bm F}^\ast (\brau  \eQ(\p^n)  ) = \brau Y \simeq \brau \eQ.\]At this point, we have an isomorphism
    \[ \pi_{\bm F}^\ast\brau_{\mathrm{sub}} (Y,\pi_{\bm F}) \simeq \brau \eQ. \]The conclusion comes from the short exact sequence (\ref{shortexactsequence}). 
    \end{proof}

\paragraph{Prediction for the constant.} Recall that the constant in \cite[conj. 3.8]{LRS} is given by the expression (\ref{cst conj}).

\begin{prop}\label{conj constante}
The constant predicted by \cite[conj. 3.8]{LRS} in the case of $\pi_{\bm F}$ is given by \begin{equation*}\label{constante}\tag{2.2}
    c_{\mathrm{LRS}}(\pi_{\bm F})  = { (n+1)^{1/2} \over (\pi d)^{3/2} }c_\infty \prod_p \left( 1 - {1 \over p} \right)^{-1/2} \left( 1 + {1 \over p} + \cdots + {1 \over p^n} \right) c_p,
\end{equation*}

\begin{equation*}\label{cinfty1} \tag{2.3} c_\infty = \mu_\infty(\{ {\bm t} \in [-1,1]^{n+1} : \vartheta_\eR(F_0(\bm t), F_1(\bm t) , F_2(\bm t)) = 1  \}),\end{equation*}
\begin{equation*}\label{cp1} \tag{2.4} c_p = \mu_p(\{ {\bm t} \in \eZ_p^{n+1} : \vartheta_p(F_0(\bm t), F_1(\bm t) , F_2(\bm t)) =  1  \}),\end{equation*}
and where we denote by $\mu_{\infty}$ the Lebesgue measure on $\eR^{n+1}$ and by $\mu_p$ the Haar measure on $\eQ_p^{n+1}$, normalised by $\mu_p(\eZ_p^{n+1} ) = 1 $. 
\end{prop}
\begin{proof}
  Since $(F_0,F_1,F_2)$ is a Birch system, the varieties $(F_i = 0)$ are geometrically irreducible as proved in \cite[§1.2]{SV}. It follows that the only codimension one points $D$ satisfying $\delta_D(\pi) < 1$ are $D_i := (F_i =~0)$ for $i \in \{0, 1 ,2\}$. The fibre $\pi_{\bm F}^{-1}(D_i)$ consists of two geometrical components interchanged by the non-trivial element of $\mathrm{Gal}(\overline{\kappa(D_i)}/ \kappa(D_i)) \simeq \eZ/ 2\eZ$, hence $\delta_{D_i}(\pi_{\bm F}) = {1 \over 2}$ for each $i \in \{0,1,2\}$, and we have $\Delta(\pi_{\bm F}) = {3 \over2}$. As in \cite[§4.7]{LRS}, using the previous notation, we thus obtain the equalities $\Gamma(\p^n, \pi_{\bm F}) =~\pi^{3 \over 2}$, $\theta(\p^n ) = {1 \over n+1}$, and $\eta(D_i) = {n+1 \over d}$ for all $i \in \{0,1,2\}$. Proposition \ref{brauer subordonné} enables us to write (see again \cite[§4.7]{LRS})
\[  \tau_{\pi_{\bm F}} (U(\eA_\eQ))^{\brausub(\p^n, \pi_{\bm F})} = {n+1 \over 2} c_\infty \prod_p \left( 1 - {1 \over p} \right)^{-1/2} \left( 1 + {1 \over p} + \cdots + {1 \over p^n} \right) c_p\]where $c_\infty$ and the constants $c_p$ are respectively given by (\ref{cinfty1}) and (\ref{cp1}). 
\end{proof}

\begin{rmk}
    The infinite product in $($\ref{constante}$)$ is convergent by \cite[th. 3.6]{LRS}.\end{rmk}

    \begin{rmk}\label{rmq ouverts}
        Here, the points of $\p^n$ under the non-split fibres are zeros of at least one of the $F_i$. The open $U_0$ in Conjecture \ref{conj lrs} is then given by 
        \[ U_0 := \p^n \smallsetminus \{ \bm y \in \p^n : F_0(\bm y) F_1(\bm y) F_2(\bm y)~=~0\}.\]Factor $($\ref{cinfty1}$)$ $($\textit{resp.} $($\ref{cp1}$))$ agrees with Conjecture \ref{conj lrs} since the set of $\bm t \in \eR^{n+1}$ $($\textit{resp.} $\eZ_p^{n+1})$ satisfying $F_0(\bm t) F_1(\bm t) F_2(\bm t) = 0$ is of measure zero. In the $p$-adic case, this follows from the equality
        \[ \mu_p (\{ \bm t \in \eZ_p^{n+1} : F_0(\bm t) F_1(\bm t) F_2(\bm t) = 0\}) = \lim_{N \to + \infty} {\#  \{ \bm t \in (\eZ \cap [0,p^N))^{n+1} : F_0(\bm t) F_1(\bm t) F_2(\bm t) = 0 \} \over p^{(n+1)N} } \]combined with \cite[lemma 5.4]{Birch}.
    \end{rmk}

\section{Application of the Selberg--Delange method}\label{sec III3}

We estimate the average of an arithmetic function in $k$ variables of the form~$\boldsymbol{\tau}_{\bm z} \ast g$, where $g$ is an arithmetic function taking small values. To achieve this goal, we use the average of $\boldsymbol{\tau}_{\bm z}$ provided by the Selberg--Delange method. Our result will be applied to establish Proposition \ref{calcul de MbmX}. It is a quite general result and it might be of use in broader contexts.\\

We fix a simply connected region $\mathcal{D}$ in $\eC^\ast$ which contains the half-line $[1, + \infty)$, and which does not contain any zero of $\zeta(s+1)$. Let $z \in \eC$. We use (see \cite[chap. II.5. p275 and theorem 5.1]{Tenenbaum}) the following Taylor series (at the origin) 
\[ \label{def ai} \tag{3.1}{(s \zeta(s+1))^z \over s+1} = \sum_{u \geqslant 0} \Gamma(z-u) a_u(z) s^u \quad \quad (\sigma > 0, z \in \eC), \]and
    \[ ( 1 - s)^z = \sum_{v \geqslant 0} b_v(z) s^v \quad \quad (\abs{s} < 1 ) .\]Let $\bm z \in \eC^k$. We define
\[ \psi(\bm s , \bm z) := (1 - s_1)^{z_1-1} \cdots (1 - s_k)^{z_k-1} \quad \quad (\forall i \in \{1, \dots , k\}, \; \abs{s_i} < 1), \]
\[ \xi(\bm s ; \bm z) := {(s_1 \zeta(s_1+1))^{z_1} \over s_1+1} \cdots {(s_k \zeta(s_k+1))^{z_k} \over s_k+1} \quad \quad (\bm s \in \mathcal{D}^k),\]  ${\bm n} ! := (n_1 !) \cdots (n_k!)$ if ${\bm n} \in \eN^k$, and more generally $\boldsymbol{\Gamma}({\bm z}) := \Gamma(z_1) \cdots \Gamma(z_k)$. \\

\begin{lemme}\label{lemme TN}
    Let ${\bm z} = (z_1, \ldots, z_k) \in \eC^k$, $g : \eN^k \to \eC$ an arithmetic function in $k$ variables and let $f=\boldsymbol{\tau}_{\bm z} \ast g$. Let $E_{{\bm z}, k} := \sum_{i=1}^k \max(1- \Re(z_i), 0)$. Assume that there exists~$\ell \in \eZ_{\geqslant 0}$ and a constant $M_{{\bm z}, \ell}(g) > 0$ such that
    \[ \sum_{{\bm d} \in \eN^k} {\abs{g({\bm d})} \over d_1\cdots d_k} (\log (\max d_i))^{\ell + 1 +E_{{\bm z}, k}} \leqslant M_{{\bm z}, \ell}(g).\]When $ X_1, \ldots, X_k  \geqslant 4$ satisfy $\min X_i \geqslant  (\max X_i)^\eta$ for some $\eta > 0$ that does not depend on the $X_i$, we have the estimate
    
    \begin{align*}
         \sum_{\substack{{\bm n} \in \eN^k \\ \forall i \in \{1, \dots, k\}, \; n_i \leqslant X_i}} \!\!\!\!\!\!\!\!\!\!\!\!\!f({\bm n}) = &\left\{ \prod_{i = 1}^k  {X_i \over (\log X_i)^{ 1 - z_i }}  \right\} \left(  \sum_{0 \leqslant \abs{\bm w} \leqslant \ell } \!\!c_{\bm w}({\bm z}; g) \!\!\prod_{1 \leqslant i \leqslant k} (\log X_i)^{- w_i}+ O_{\eta,\ell,\bm z} \left( { M_{{\bm z}, \ell}(g) \over (\log \min X_i )^{\ell+ 1 } } \right) \right),
    \end{align*}
    with $\displaystyle c_{\bm w}({\bm z}; g) = \displaystyle {(-1)^{\abs{\bm w}} \over \boldsymbol{\Gamma}({\bm z})} {\partial^{\abs{\bm w}} \psi(\bm s, \bm z) \over \displaystyle \partial^{w_1} s_1 \cdots \partial^{w_k} s_k} (\mathbf{0}) \displaystyle {\displaystyle \partial^{\abs{\bm w}} (\xi(\bm s ; \bm z) D_g(\mathbf{1 - s}))  \over \displaystyle \partial^{w_1} s_1 \cdots \partial^{w_k} s_k} (\mathbf{0})$.\\
    In particular, we have $c_0({\bm z}; g)~=~D_g(\mathbf{1})\underset{i = 1}{\overset{k}{\prod}}~a_0(z_i)$.
\end{lemme}
\begin{proof}
    We write \begin{align*} \label{3.1}
         \sum_{\substack{ \bm n \in \eN^k \\ \forall i \in \{1, \dots, k\}, \; n_i \leqslant X_i}} \!\!\!\!\!\!\! f({\bm n}) & = \!\!\!\!\!\!\! \sum_{\substack{ \bm d \in \eN^k \\ \forall i \in \{1, \dots, k\}, \; d_i \leqslant X_i}} \!\!\!\!\!\!\! g({\bm d}) \left( \sum_{\substack{\bm m \in \eN^k \\ \forall i \in \{1, \dots, k\}, \; m_i \leqslant X_i / d_i}} \!\!\!\!\!\!\! \boldsymbol{\tau}_{\bm z} ( {\bm m}) \right). \tag{3.2}
    \end{align*}
 
  We now use the well-known estimate
    \begin{equation*}\label{selbergdelange}\tag{3.3}\sum_{\substack{ \bm m \in \eN^k \\ \forall i \in \{1 , \dots, k\}, \; m_i \leqslant X_i / d_i} } \!\!\!\!\!\!\!\!\!\!\!\!\!\!\boldsymbol{\tau}_{\bm z} ( {\bm m}) \ll {X_1 \cdots X_k \over d_1 \cdots d_k} \prod_{i = 1}^k \left( \log \left( {2X_i \over d_i} \right)\right)^{\Re (z_i) - 1},\end{equation*}which is valid whenever $d_i \leqslant X_i$ for all $i \in \{1, \dots, k\}$. We deduce 
    \begin{align*}
        \sum_{\substack{ \bm d \in \eN^k \\ \forall i \in \{1, \dots, k\}, \; d_i \leqslant X_i \\ \exists i_0 \; : \;  d_{i_0} > \sqrt{X_{i_0}} }} \!\!\!\!\!\!\!\!\!\!\! g({\bm d}) \left( \sum_{\substack{\bm m \in \eN^k \\ \forall i \in \{1, \dots, k\}, \; m_i \leqslant X_i / d_i}} \!\!\!\!\!\!\!\!\!\!\!\!\!\!\! \boldsymbol{\tau}_{\bm z} ( {\bm m}) \right) & \ll   X_1 \cdots X_k  \!\!\!\!\!\!\!\!\!\!\!\!\!\!\!\! \sum_{\substack{ \bm d \in \eN^k \\ \forall i \in \{1, \dots, k\}, \; d_i \leqslant X_i \\ \exists i_0 \; : \;  d_{i_0} > \sqrt{X_{i_0}} }} \!\!\!\!\!\!\!\!\!\!\!\!\! {\abs{g({\bm d})} \over  d_1 \cdots d_k }\prod_{i = 1}^k \left( \log \left( {2X_i \over d_i} \right)\right)^{\Re (z_i) - 1}  .
    \end{align*}We can now bound each factor $\left( \log \left( {2X_i \over d_i} \right)\right)^{\Re (z_i) - 1} $ by $(\log X_i)^{\Re(z_i) - 1 + \max (0, 1 - \Re(z_i))}$, so that
      \begin{align*}
        \sum_{\substack{ \bm d \in \eN^k \\ \forall i \in \{1, \dots, k\}, \; d_i \leqslant X_i \\ \exists i_0 \; : \;  d_{i_0} > \sqrt{X_{i_0}} }} \!\!\!\!\!\!\!\!\!\!\!\!\!\!  g({\bm d}) \left( \sum_{\substack{\bm m \in \eN^k \\ \forall i \in \{1, \dots, k\}, \; m_i \leqslant X_i / d_i}} \!\!\!\!\!\!\!\!\!\!\!\!\!\! \boldsymbol{\tau}_{\bm z} ( {\bm m}) \right) & \ll \left\{  \prod_{i = 1}^k X_i  \log\left( X_i \right)^{\Re (z_i) - 1} \right\} (\log \max X_i)^{E_{\bm z,k}}  \!\!\!\!\!\!\!\!\!\!\!\!\!\!\!\! \sum_{\substack{ \bm d \in \eN^k \\ \forall i \in \{1, \dots, k\}, \; d_i \leqslant X_i \\ \exists i_0 \; : \;  d_{i_0} > \sqrt{X_{i_0}}}} \!\!\!\!\!\!\!\!\!\!\!\! {\abs{g({\bm d})} \over  d_1 \cdots d_k } .
    \end{align*}
    Let $h$ be an integer such that $0 \leqslant h \leqslant \ell + 1 + E_{\bm z, \ell}$. For $y \geqslant 1$, we have the trivial bound
    \begin{equation*}
        \label{inegtriv}\tag{3.4}
        \forall d \geqslant y, \; (\log d)^h \leqslant \left( {\log 3d \over \log 3y } \right)^{\ell + 1 + E_{\bm z, \ell}} (\log 3y)^h,
    \end{equation*}which implies for $h= 0$, $d = d_{i_0}$ and $y =\sqrt{X_{i_0}}$, the bound
    \begin{align*} \sum_{\substack{ \bm d \in \eN^k \\ \forall i \in \{1, \dots, k\}, \; d_i \leqslant X_i \\ \exists i_0 : \; d_{i_0} > \sqrt{X_{i_0}} }}  {\abs{g({\bm d})} \over  d_1 \cdots d_k } & \ll_\eta {M_{{\bm z},\ell} (g)\over (\log (\min X_i))^{\ell  + 1 + E_{{\bm z}, k}} }.
    \end{align*}
    It follows that
    \begin{align*}
        \sum_{\substack{ \bm d \in \eN^k \\ \forall i \in \{1, \dots, k\}, \; d_i \leqslant X_i \\ \exists i_0 \; : \;  d_{i_0} > \sqrt{X_{i_0}}}}  g({\bm d}) \left( \sum_{\substack{\bm m \in \eN^k \\ \forall i \in \{1, \dots, k\}, \; m_i \leqslant X_i / d_i}} \!\!\!\!\!\! \boldsymbol{\tau}_{\bm z} ( {\bm m}) \right)\ll_\eta \left\{ \prod_{i = 1}^k  X_i (\log X_i)^{\Re(z_i) - 1 }  \right\} {M_{\bm z, \ell}(g) \over (\log \min X_i)^{\ell + 1} }, 
    \end{align*}
    which is an admissible error term. We now focus on the complementary contribution, that is the contribution of the $\bm d$ such that $d_i  \leqslant  \sqrt{X_i}$ for all $i \in \{1, \dots, k\}$. We can apply the Selberg--Delange method \cite[th. 5.2]{Tenenbaum} which yields the estimate
    \[  \sum_{m \leqslant x} \tau_z (m) = x  \sum_{u = 0}^{\ell} a_u(z) (\log x)^{z - 1 - u} + O\left(x (\log x)^{\Re(z) - 2 - \ell}\right) \quad \quad (x \geqslant 2). \]Thus, we have 
    \begin{align*}
        \sum_{\substack{{\bm m} \in \eN^k \\ \forall i \in \{1, \dots, k\}, \; m_i \leqslant X_i / d_i}} \!\!\!\!\!\!\!\! \boldsymbol{\tau}_{\bm z} ( {\bm m}) = & \prod_{i = 1}^k \left( {X_i \over d_i}  \sum_{u_i = 0}^\ell a_{u_i}(z_i) (\log (X_i/d_i))^{z_i - 1 - u_i} + O\left({X_i \over d_i} (\log (X_i / d_i))^{\Re(z_i) - 2 - \ell }\right) \right).
        \end{align*}
        Developing the product then leads to the following estimate which is valid for $d_i \leqslant \sqrt{X_i}$,
        \begin{align*}
       \sum_{\substack{{\bm m} \in \eN^k \\ \forall i \in \{1, \dots, k\}, \;  m_i \leqslant X_i / d_i}}\boldsymbol{\tau}_{\bm z} ( {\bm m}) = & \left\{ \prod_{i = 1}^k {X_i \over d_i}  \right\} \sum_{\substack{ {\bm u} \in \eZ^k \\  |\bm u | \leqslant \ell}} \prod_{i =1}^k a_{u_i}(z_i) (\log (X_i/d_i))^{z_i - 1 - u_i} \\ &+ O_{\ell,\bm z} \left( \left\{ \prod_{i = 1}^k {X_i \over d_i } (\log (X_i/d_i))^{\Re(z_i) - 1} \right\} \log \left( \min X_i \right)^{-(\ell+1)} \right).
        \end{align*}
        Since $d_i \leqslant \sqrt{X_i}$ for all $i \in \{ 1, \dots, k\}$, we can write the quantity $ (\log (X_i/d_i))^{z_i-j}  $ as
    \begin{align*} \tag{3.5} \label{dev log}   (\log X_i)^{z_i-j}  \left( 1 - {\log d_i \over \log X_i} \right)^{z_i-j} \!\!\!\!\!\!\!\!\!  =  (\log X_i)^{z_i-j}\left( \sum_{0 \leqslant v \leqslant \ell-j+1} \!\!\!\!\!\!\! b_v(z_i-j) \left( {\log d_i \over \log X_i} \right)^v \!\!\!\! + O\left( {\log d_i \over \log X_i} \right)^{\ell-j + 2} \right).
    \end{align*}
        Using (\ref{dev log}) for $j = 1+u_i$ yields
        \begin{align*} 
       \sum_{\substack{{\bm m} \in \eN^k \\ \forall i \in \{1, \dots, k\},\;  m_i \leqslant X_i / d_i}}\!\!\!\!\!\!\! \boldsymbol{\tau}_{\bm z} ( {\bm m}) = & \left\{ \prod_{i = 1}^k {X_i \over d_i} (\log X_i)^{z_i - 1 }  \right\} \sum_{ \substack{{\bm u,\bm v} \in \eZ^k \\ \abs{{\bm u + \bm v}}\leqslant \ell }} \prod_{i =1}^k a_{u_i}(z_i)   b_{v_i}(z_i - 1 - u_i) {(\log d_i)^{v_i} \over  (\log X_i)^{ u_i+v_i}} \\
        &+ O_{\ell, \bm z} \left( \left\{ \prod_{i = 1}^k {X_i  \over d_i } (\log X_i)^{\Re(z_i) - 1} \right\}   {(\log \max d_i)^{\ell +1+ E_{\bm z, k} } \over (\log \min X_i)^{  \ell+1 }} \right) .
    \end{align*}

   \noindent Let ${\bm w} = (w_1,\ldots, w_k)$ with $\abs{\bm w}  \leqslant \ell$, and 
    \[ \lambda_{\bm w}( {\bm z,\bm d}):= \sum_{ \bm u+\bm v=\bm w } \; \prod_{i =1}^k a_{u_i}(z_i)   b_{v_i}(z_i - 1 - u_i) (\log d_i)^{v_i}.\] At this point, we have the estimate

    \begin{align*}
         \sum_{\substack{{\bm m} \in \eN^k \\ \forall i\in \{1, \dots, k\}, \; m_i \leqslant X_i / d_i}} \!\!\!\!\!\!\!\!\!\!\!\!\!\! \boldsymbol{\tau}_{\bm z} ( {\bm m}) = & \left\{ \prod_{i = 1}^k {X_i \over d_i} (\log X_i)^{z_i - 1 }  \right\}  \sum_{ \abs{\bm w} \leqslant \ell } \lambda_{\bm w}( {\bm z,\bm d}) \prod_{i=1}^k (\log X_i)^{- w_i} \\
         & +O_{\ell, \bm z} \left( \left\{ \prod_{i = 1}^k {X_i \over d_i } (\log X_i)^{\Re(z_i) - 1} \right\} { (\log \max d_i)^{\ell+1+E_{\bm z,k}} \over (\log \min X_i)^{ \ell +1}} \right) .
    \end{align*}
    
    \noindent Inserting this expression in (\ref{3.1}) yields
    
    \begin{small}
    \begin{align*}
         \sum_{\substack{{\bm n} \in \eN^k \\  \forall i \in \{1, \dots, k\}, \; n_i \leqslant X_i}} \!\!\!\!\!\!\! f({\bm n}) = & \left\{ \prod_{i = 1}^k  X_i (\log X_i)^{z_i - 1 }  \right\}  \left(  \underset{i = 1}{\overset{k}{\prod}} a_0(z_i) \!\!\!\!\!\!\!\!\!\!  \sum_{\substack{{\bm d} \in \eN^k \\ \forall i \in \{1, \dots, k\}, d_i \leqslant  \sqrt{X_i} }}  \!\! \! \! \! \!  \!  {g({\bm d}) \over d_1 \cdots d_k} \right. \\
         & \left. + \sum_{ 1 \leqslant \abs{\bm w} \leqslant \ell } \! \! \! \! \! \! \! \! \! \!  \! \sum_{\substack{{\bm d} \in \eN^k \\ \forall i \in \{1, \dots, k\}, d_i \leqslant  \sqrt{X_i} }} \! \! \! \! \!\!\! \! \! \! \!  \!  \! { g({\bm d})  \lambda_{\bm w}( {\bm z,\bm d}) \over d_1 \cdots d_k} \prod_{1 \leqslant i \leqslant k} (\log X_i)^{- w_i}  + O_{\eta,\ell, \bm z}\left( { M_{\bm z, \ell} (g) \over (\log \min X_i)^{ \ell +1}} \right)  \right).
    \end{align*}
    \end{small}We shall remark that the sums over ${\bm d}$ are convergent by assumption on $g$. Moreover, the bound
     \[ \prod_{i = 1}^k (\log d_i)^{w_i} \ll \left( { \log(  \max d_i ) \over \log \min X_i } \right)^{\ell+ 1 + E_{\bm z,k}} (\log  \min X_i)^{\abs{\bm w} } \] ensures that the contribution of the triples $\bm d$ for which there exists $d_i >   \sqrt{X_i}$ contributes to the error term. Finally, we write
\begin{align*} 
c_{\bm w}({\bm z};g) & := \sum_{ \bm u+\bm v = \bm w } \prod_{i =1}^k a_{u_i}(z_i)   b_{v_i}(z_i - 1 - u_i) \sum_{{\bm d} \in \eN^k} { g({\bm d})   \over d_1 \cdots d_k} (\log d_1)^{v_1} \cdots (\log d_k)^{v_k} \\
& = \sum_{ \bm u+\bm v = \bm w } {\partial^{\abs{\bm v}} D_g \over \partial^{v_1} s_1 \cdots \partial^{v_k} s_k} ( \mathbf{1} )  \prod_{i =1}^k a_{u_i}(z_i)   b_{v_i}(z_i - 1 - u_i)
\end{align*}
and use the equalities  
\[ 
    {1 \over {\bm v}!} {\partial^{\abs{\bm v}} \psi(\bm s, \bm z)\over \partial^{v_1} s_1 \cdots \partial^{v_k} s_k} (\mathbf{0})  = \prod_{i = 1}^k  b_{v_i}(z_i - 1) \quad \text{ and }\quad 
    {1 \over {\bm u}!} {\partial^{\abs{\bm u}} \xi(\bm s ; \bm z) \over \partial^{u_1} s_1 \cdots \partial^{u_k} s_k} (\mathbf{0})  = \prod_{i = 1}^k \Gamma(z_i - u_i) a_{u_i}(z_i)
\]to conclude that
  \begin{align*} 
c_{\bm w}({\bm z};g) & = \sum_{ \bm u+\bm v = \bm w } {1 \over {\bm u}!}{1 \over {\bm v}!}  {\partial^{\abs{\bm v}} D_g \over \partial^{v_1} s_1 \cdots \partial^{v_k} s_k} ( \mathbf{1} )  {\partial^{\abs{\bm u}}\xi(\bm s ; \bm z) \over \partial^{u_1} s_1 \cdots \partial^{u_k} s_k} (\mathbf{0}) \left( {1 \over \boldsymbol{\Gamma}(\mathbf{z-u})} {\partial^{\abs{\bm v}} (1 - \mathbf{s})^{\mathbf{z-u} - 1} \over \partial^{v_1} s_1 \cdots \partial^{v_k} s_k} (\mathbf{0}) \right)\\
& = {1 \over \boldsymbol{\Gamma}({\bm z})} {\partial^{\abs{\bm w}} \psi(\bm s, \bm z) \over \partial^{w_1} s_1 \cdots \partial^{w_k} s_k} (\mathbf{0}) \sum_{ \bm u+\bm v = \bm w } {(-1)^{\abs{\bm u}} \over {\bm u}!}{1 \over {\bm v}!}  {\partial^{\abs{\bm v}} D_g \over \partial^{v_1} s_1 \cdots \partial^{v_k} s_k} ( \mathbf{1} )  {\partial^{\abs{\bm u}} \xi(\bm s ; \bm z) \over \partial^{u_1} s_1 \cdots \partial^{u_k} s_k} (\mathbf{0}) \\
& = {(-1)^{\abs{\bm w}} \over \boldsymbol{\Gamma}({\bm z})} {\partial^{\abs{\bm w}}  \psi(\bm s, \bm z) \over \partial^{w_1} s_1 \cdots \partial^{w_k} s_k} (\mathbf{0}) {\partial^{\abs{\bm w}} (D_g(\mathbf{1 - s})\xi(\bm s ; \bm z)) \over \partial^{w_1} s_1 \cdots \partial^{w_k} s_k} (\mathbf{0}). 
\end{align*}
This completes the proof of Lemma \ref{lemme TN}.
\end{proof}

\section{Proportion of diagonal plane conics with a rational point and coefficients in arithmetic progression}\label{sec III4}

Let $q \geqslant 2$ and ${\bm n} \in (\eZ \cap [0,q))^3$ such that $\gcd(n_0,n_1,n_2,q)= 1$. For $\bm t \in \eN^3$, let $Q_{\bm t}$ be the quadratic form defined by $Q_{\bm t} (y_0,y_1,y_2)=  t_0 y_0^2 + t_1 y_1^2 - t_2 y_2^2$. For $B \geqslant 2$ and $\bm B = (B_0,B_1,B_2) \in \left[ {B \over (\log B)^2}, B \right]^3$, we introduce the quantity \[ N({\bm n},q,\bm B) := \# \left\{  {\bm t} \in \eN^3 : \begin{tabular}{c} $\gcd (t_0,t_1,t_2) = 1, \; \forall i \in \{0,1,2\}, \; t_i \leqslant B_i$,\\
$ Q_{\bm t}\text{ has a $\eQ$-point}, \; \bm t \equiv \bm n \Mod{q}$ \end{tabular}\right\}.\]We aim to estimate $N({\bm n},q,\bm B)$ as $B$ goes to $+ \infty$, uniformly in $q$, when $q$ is bounded by a power of $\log B$. We recall that the notation $\vartheta_p = \vartheta_{\eQ_p}$ is defined in (\ref{notation}).

\begin{rmk}\label{rmq cruciale}
     Let $q$ be a positive integer divisible by $8$. Let $ \bm n = (n_0,n_1,n_2) \in (\eZ \cap [0,q))^3$ and~$\bm t = (t_0, t_1,t_2) \in \eN^3$ such that $\gcd(t_0,t_1,t_2) = 1$ and $\bm t \equiv \bm n \Mod{q}$. For any prime divisor~$p$ of $q$, \cite[chap. 3. th. 1]{Serre2} ensures that if $\max (v_p(n_0n_2),v_p(n_1n_2)) < v_p(q) $  when $p$ is odd, (\textit{resp.} $\max (v_2(n_0n_2),v_2(n_1n_2)) < v_2(q) - 2$), then the quantity~$\vartheta_p(\bm t) $ (\textit{resp.} $\vartheta_2(\bm t)$) only depends on $\bm n$. In this case, we write~$\vartheta_p(\bm t) = \vartheta_p(\bm n)$. 
\end{rmk}  Therefore, if $8$ divides $q$, the quantity
\begin{equation*}\label{facteurs delta}\tag{4.1}
  \delta_p(\bm n) := \left\{ \begin{tabular}{cc}
      $ \vartheta_p(\bm n) $ & if $p \mid q$, \\
      \\
    $\displaystyle {\left(1 + {1 \over p} + {1 \over p^2}\right)\left(1 + {1 \over 2p} + {1 \over p^2}  \right) \over \left(1 - {1 \over p^2}\right)^2}$  & otherwise,
   \end{tabular} \right.
\end{equation*}
is well-defined for all $\bm n \in (\eZ / q \eZ)^3$ whose representative in $[0,q)^3$, still denoted $\bm n$, satisfies for any $p \mid q$ the condition $\max (v_p(n_0n_2),v_p(n_1n_2)) < v_p(q) $ when $p$ is odd, (\textit{resp.} $\max (v_2(n_0n_2),v_2(n_1n_2)) < v_2(q)- 2$). The choice of the value of $\delta_p(\bm n)$ for $p \nmid q$ comes from the value of the factor corresponding to $p$ in the Euler product in \cite[th. 1.1]{LRS}.

\begin{tho}\label{problème de comptage}
         Let $q \in \eN$ be a square-full integer divisible by $8$ and $\bm n \in [0,q)^3$ such that $\gcd(\bm n, q)=1$. Let $B \geqslant 16$, and $\bm B = (B_0,B_1,B_2) \in \left[ {B \over (\log B)^2}, B \right]^3$. If $q \leqslant (\log B)^{1/7}$, $\max (v_2(n_0n_2),v_2(n_1n_2)) < v_2(q) - 2$ and $\max (v_p(n_0n_2),v_p(n_1n_2)) < v_p(q)$ for all odd prime $p$ dividing $q$, then we have the estimate
    \begin{align*} 
    N({\bm n},q,\bm B) & = {2 \over \pi^{3/2}\varphi(q)^3}\left( \prod_{p \in \cP}\left(1 - {1 \over p} \right)^{3/2}  \delta_p(\bm n) \right) \prod_{0 \leqslant i \leqslant 2} {B_i \over (\log B_i)^{1/2}} + O \left( {B_{012} (\log_2 B)q^3 \over \varphi(q)^3 (\log B)^{5/2}} \right),
    \end{align*}
    and the implied constant does not depend on $q$ and $\bm n$.
    \end{tho}

The end of this section and the next section are dedicated to the proof of Theorem \ref{problème de comptage}.

\paragraph{Step 1 : preliminary calculations.} We follow the strategy of \cite[§5.8]{LRS}. We denote by~$C_{\bm m,\bm \ell}$ the conic of equation \[ m_{12}\ell_0 Y_0^2 + m_{02}\ell_1 Y_1^2 = m_{01}\ell_2Y_2^2.\]Recall that the notation $\ell_{012}$ stands for $ \ell_0\ell_1\ell_2$ and $m_{012}$ stands for $  m_{01}m_{02}m_{12}$.\\

For ${\bm X} = (X_0,X_1,X_2)\in \eN^3$, ${\bm b} = (b_0,b_1,b_2) \in \eN^3$, ${\bm m}= (m_{01},m_{02},m_{12}) \in \eN^3$, and ${\bm r} = (r_0,r_1,r_2) \in (\eZ/q\eZ)^3$, we introduce the quantity
    \[ \cN_{\bm b,\bm m,\bm r} ({\bm X}) := \sum_{\substack{ \forall i \in \{0,1,2\}, \;\ell_i \leqslant X_i \\ \gcd(\ell_i,b_j,b_k) =  \gcd(\ell_{012},m_{012}) = 1 \\  \bm \ell \equiv \bm r \Mod{q} } } \!\!\!\!\! \mu^2(\ell_{012}) \prod_{p \nmid q} \vartheta_p(m_{12}\ell_0, m_{02} \ell_1, m_{01} \ell_2).\]We start by writing $N(\bm n, q,\bm B)$ using $\cN_{\bm b,\bm m,\bm r} $.

\begin{lemme}\label{réécriture Nn(B)}
Let $C > 0$. We have $N(\bm n, q, \bm B) = 0$ if $\displaystyle \prod_{p \mid q} \vartheta_p(\bm n) = 0$. Otherwise, $N({\bm n},q,\bm B) $ is equal to
    \begin{align*}
          \sum_{\substack{{\bm b}, \bm m \in \left[1,(\log B)^C\right]^3 \\\gcd(m_{ij},b_k)=1 \\ \gcd(b_0,b_1,b_2) =1} } \sum_{\substack{{  \bm r}\in (\eZ/q \eZ)^3\\ b_i^2 m_{ij} m_{ik} r_i \equiv n_i \Mod{q} } } \!\!\!\!\!\!\! \mu^2(m_{012})& \cN_{\bm b,\bm m,\bm r} \left( {B_0 \over b_0^2 m_{01}m_{02}}, {B_1 \over b_1^2 m_{01}m_{12}},  {B_2 \over b_2^2 m_{12}m_{02}} \right)\\
        &+ O_C\left( {q^3 B_{012}  \over (\log B)^C}  \right).
    \end{align*} 
\end{lemme}
\begin{proof}
Using the Hasse--Minkowski theorem \cite[ch. 3, th. 8]{Serre2}, we write 
\[ \vartheta_\eQ(\bm t) = \vartheta_\eR (\bm t) \prod_{p } \vartheta_p(\bm t) = \prod_{p } \vartheta_p(\bm t),\]where the last equality comes from the positivity of $t_0,t_1,t_2$. Since the integers $t_i$ are coprime, we can use Remark \ref{rmq cruciale} which enables us to isolate the solubility in $\eQ_p$ for all $p \mid q$, and which gives the first part of the lemma. Let 
\[ N'({\bm n},q,\bm B) := \# \left\{ {\bm t} \in \eN^3 : \begin{tabular}{l} $\forall i \in \{0,1,2\}, \; t_i \leqslant B_i$, $\gcd(t_0,t_1,t_2) = 1$, \\  $\forall p \nmid q, \; \vartheta_p(\bm t) = 1$, ${\bm t} \equiv {\bm n} \; (\bmod \; q)$ \end{tabular} \right\}. \]Since $\max (v_2(n_0n_2),v_2(n_1n_2)) < v_2(q) - 2$ and $\max (v_p(n_0n_2),v_p(n_1n_2)) < v_p(q)$ for all odd $p$ dividing~$q$, the solubility in~$\eQ_p$ for all $p \mid q$ only depends on $\bm n$. It follows that 
\[ N({\bm n},q,\bm B) =  \left(\prod_{p \mid q} \vartheta_p(\bm n) \right) N'({\bm n},q,\bm B) .\]
 For $i \in \{0,1,2\}$, the integer $t_i$ can be uniquely written $t_i = b_i^2 c_i$ with $b_i, c_i \in \eN$ and $\mu^2(c_i) = 1$. We can then decompose $N'({\bm n},q,\bm B) $ as
\begin{align*}
   N'({\bm n},q,\bm B) & = \sum_{ \substack{ \forall i \in \{0,1,2\}, \; t_i \leqslant B_i   \\ \gcd(t_0,t_1,t_2) = 1  \\ {\bm t } \equiv {\bm n} \Mod{q}} } \prod_{p \nmid q} \vartheta_p(\bm t)\\
   & = \sum_{ \substack{ \forall i \in \{0,1,2\}, \; b_i \leqslant \sqrt{B_i}  }}  \sum_{\substack{{ \bm v}\in (\eZ/q \eZ)^3 \\
   b_i^2 v_i \equiv n_i \Mod{q}}} \sum_{ \substack{ \forall i \in \{0,1,2\}, \;  c_i \leqslant B_i/b_i^2   \\ \gcd(b_0c_0,b_1c_1,b_2c_2) = 1 \\ {\bm c} \equiv {\bm v} \Mod{q} } } \!\!\!\!\!\!\!\!\!\!\!\!\!\!\! \mu^2(c_0)\mu^2(c_1)\mu^2(c_2) \prod_{p \nmid q} \vartheta_p(\bm c),
\end{align*}
where the last equality comes from the fact that the number of $\eQ$-points of a conic only depends on the square-free part of its coefficients. We now handle the contribution of the terms for which the $b_i$ are sufficiently large. Let $C > 0$. If there exists $ i_0 \in \{0,1,2\}$ such that $b_{i_0} > (\log B)^C$. We have the upper-bound
\[ 
    \sum_{ \substack{ \bm c \in \eN^3  \\ \gcd(b_0c_0,b_1c_1,b_2c_2) = 1 \\ \forall i \in \{0,1,2\}, \; c_i \leqslant B_i/b_i^2 } } \!\!\!\!\!\!\!\!\!\!\!\!\!\!\! \mu^2(c_0)\mu^2(c_1)\mu^2(c_2) \prod_{p \nmid q} \vartheta_p(\bm c) \leqslant  {B_{012} \over b_0^2 b_1^2 b_2^2}\]so that
    \begin{align*}
        \sum_{ \substack{ \bm b \in \eN^3  \\ b_{i_0} > (\log B)^C } } \sum_{ \substack{ \bm c \in \eN^3  \\ \gcd(b_0c_0,b_1c_1,b_2c_2) = 1 \\ \forall i \in \{0,1,2\}, \; c_i \leqslant B_i/b_i^2 } } \!\!\!\!\!\!\!\!\!\!\!\!\!\!\! \mu^2(c_0)\mu^2(c_1)\mu^2(c_2)\prod_{p \nmid q} \vartheta_p(\bm c)&  \leqslant  B_{012}\!\!\!\!\!\! \sum_{ \substack{ \bm b \in \eN^3  \\  b_{i_0} > (\log B)^C } }  {1  \over b_0^2 b_1^2 b_2^2} \ll {B_{012} \over (\log B)^C}.
    \end{align*} Therefore, we deduce that
    \begin{align*}
        N'({\bm n},q,\bm B) = & \sum_{ \substack{ \bm b \in   \left[1, (\log B)^C\right]^3 }}  \sum_{\substack{{ \bm v}\in (\eZ/q \eZ)^3 \\
   b_i^2 v_i \equiv n_i \Mod{q}}} \sum_{ \substack{   c_i \leqslant B/b_i^2   \\ \gcd(b_0c_0,b_1c_1,b_2c_2) = 1 \\ {\bm c} \equiv {\bm v} \Mod{q} } } \!\!\!\!\!\!\!\!\!\!\!\!\!\!\!\mu^2(c_0)\mu^2(c_1)\mu^2(c_2) \prod_{p \nmid q} \vartheta_p(\bm c) \\ 
        & + O\left( {q^3B_{012} \over (\log B)^C}  \right).
    \end{align*} 
    We now take into account the common factors of $c_0,c_1$ and $c_2$. For $i\neq j$, we introduce the integers $m_{ij} := \gcd(c_i,c_j)$. Since the integers $c_i$ are coprime, there exists $(\ell_0,\ell_1,\ell_2) \in \eN^3$ such that $c_0 = m_{01}m_{02}\ell_0$, $c_1 = m_{01}m_{12}\ell_1$ and $c_2 = m_{02}m_{12}\ell_2$. Moreover, the $\ell_i$ and the $m_{ij}$ are also coprime. We can thus decompose the sum over the triples $(c_0,c_1,c_2)$ as follows
    \begin{align*}
        & \sum_{ \substack{ \forall i \in \{0,1,2\}, \; c_i \leqslant B_i/b_i^2  \\ \gcd(b_0c_0,b_1c_1,b_2c_2) = 1 \\  {\bm c} \equiv {\bm v} \Mod{q} } } \!\!\!\!\!\!\!\!\!\!\!\! \mu^2(c_0)\mu^2(c_1)\mu^2(c_2) \prod_{p \nmid q} \vartheta_p(\bm c) \\ & = \sum_{\substack{\forall \{i,j,k\} = \{0,1,2\}, \; m_{ij} \leqslant B_k \\ \gcd(m_{ij},b_k)  = 1  } } \sum_{\substack{{\bm r} \in (\eZ/q\eZ)^3 \\ m_{ij}m_{ik}r_i\equiv v_i \Mod{q} \\
         }} \!\!\!\!\!\!\!\! \mu^2(m_{012}) \!\!\!\!\!\!\!\sum_{\substack{{\bm \ell}=(\ell_0,\ell_1,\ell_2) \in \eN^3 \\ \gcd(\ell_i,b_j,b_k) =  \gcd(\ell_{012},m_{012}) = 1 \\ \ell_i m_{ij}m_{ik} \leqslant B_i/b_i^2 \\ \bm \ell \equiv \bm r \Mod{q}  } } \!\!\!\!\!\!\!\!\!\!\!\!\!\!\! \mu^2(\ell_{012}) \prod_{p \nmid q} \vartheta_p(\bm Y_{\bm m, \bm \ell}),
    \end{align*} 
    where $\bm Y_{\bm m, \bm \ell} = ( m_{01}m_{02}\ell_0, m_{01}m_{12}\ell_1,  m_{02}m_{12}\ell_2) $, and where the conditions under the sums hold for \\ $\{i,j,k\}~=~\{0,1,2\}$. Using the change of variables \[ \label{4.7.1} \tag{4.2} \left\{ \begin{tabular}{l} $Y_0 := m_{01}m_{02} x_0$ \\ $Y_1 := m_{12}m_{01}x_1$ \\$Y_2 := m_{12}m_{02}x_2 $,\end{tabular} \right.  \] we transform the conic of equation $m_{01}m_{02}\ell_0x_0^2 + m_{01}m_{12}\ell_1 x_1^2 =  m_{02}m_{12}\ell_2 x_2^2$ into the conic of equation $m_{12}\ell_0 Y_0^2 + m_{02}\ell_1 Y_1^2 = m_{01}\ell_2 Y_2^2$. We thus obtain that $N'({\bm n},q,\bm B)$ is equal to
     \begin{align*}
         \sum_{\substack{{\bm b}, \bm m \in \left[1,(\log B)^C\right]^3 \\\gcd(m_{ij},b_k)=1 \\ \gcd(b_0,b_1,b_2) =1} } \sum_{\substack{{  \bm r}\in (\eZ/q \eZ)^3\\ b_i^2 m_{ij} m_{ik} r_i \equiv n_i \Mod{q} } }    \!\!\!\!\!\!\!\!\!\!\!\!  \mu^2(m_{012}) & \cN_{\bm b,\bm m,\bm r} \left( {B_0 \over b_0^2 m_{01}m_{02}}, {B_1 \over b_1^2 m_{01}m_{12}},  {B_2 \over b_2^2 m_{12}m_{02}} \right) \\
        & + O\left( { q^3 B_{012} \over (\log B)^C}  \right),
    \end{align*} 
    where the condition $b_i^2 m_{ij}m_{ik}r_i = n_i \Mod{q}$ under the second sum holds for all $ \{i,j,k\} = \{0,1,2\}$. Note that, as already done for the index ${\bm b}$, we neglected the contribution of the triples ${\bm m}$ which have a coordinate $> (\log B)^C$. Indeed, since we have the trivial bound $\cN_{\bm b,\bm m,\bm r}({\bm X}) \leqslant X_0 X_1X_2$ the same computation leads to a contribution of size~$O\left( {q^3 B_{012} \over (\log B)^C}\right)$. This concludes the proof of the lemma.
\end{proof}

\paragraph{Step 2 : estimating $\cN_{\bm b,\bm m,\bm r}({\bm X})$.} For $a,b,c \in \eN$, we define
      \begin{equation*} \label{theta} \tag{4.3}\Theta(a,b,c) :=  \sum_{\substack{{\bm u}, {\bm u'} \in \eN^3  \\
          (\ref{cond theta}) }}  \left( {bc \over u_0}\right) \left( {ac \over u_1}  \right) \left( {-ab \over u_2} \right), \end{equation*}where we refer to the conditions
          \[ \tag{4.4}\label{cond theta} \left\{ \begin{tabular}{l}
              $u_0u_1u_2 \neq 1$, \; $u_0'u_1'u_2' \neq 1$, \\
              $u_0u_0' = {a \over \gcd(a,q)}$, $u_1u_1' = {b \over \gcd(b,q)}$, $u_2u_2' = {c \over \gcd(c,q)}$,\\
              $\gcd(u_{012}u_{012}', q) = 1$.
          \end{tabular} \right.\]Let $\bm X \in [2, + \infty)^3$, $\bm m$, $\bm b \in \eN^3$ and $ \bm r \in (\eZ/ q \eZ)^3$ satisfying

    \begin{equation*} \tag{4.5}\label{H1}
        \mu^2(m_{01} m_{02} m_{12} ) = 1,
    \end{equation*}
    
    \begin{equation*} \label{H2} \tag{4.6} \gcd(b_0,b_1,b_2) = \gcd(m_{01},b_2) = \gcd(m_{02},b_1) = \gcd(m_{12},b_0) = 1,
    \end{equation*}

    \begin{equation*} \label{congr} \tag{4.7} 
    b_i^2 m_{ij}m_{ik}r_i \equiv n_i \Mod{q},
    \end{equation*}
    
    \begin{equation*} \label{H3} \tag{4.8}\max( b_0,b_1,b_2,m_{12},m_{02},m_{01}) \leqslant  (\log \min X_i)^C,
    \end{equation*}and 
    
    \begin{equation*} \label{hyp odg Xi} \tag{4.9}  \min X_i \geqslant  (\max X_i)^\eta 
     \end{equation*}
     for some $\eta > 0$ that does not depend on the $X_i$. We introduce, for $\bm r \in (\eZ/q \eZ)^3$, the quantities
\begin{equation*}\label{M}\tag{4.10}
        \cM_{\bm b,\bm m,\bm r}({\bm X}) := \sum_{\substack{ \forall i \in \{0,1,2\}, \; \ell_i \leqslant X_i \\ \gcd(\ell_i,b_j,b_k) =  \gcd(\ell_{012},m_{012}) = 1 \\   \bm \ell \equiv \bm r \Mod{q}} } {\mu^2(\ell_{012}) \tau( \gcd(\pimp{(\ell_{012})},q) ) \over \tau(\pimp{(\ell_{012})})},
    \end{equation*}
    where the condition $ \gcd(\ell_i,b_j,b_k) =   1 $ is satisfied for every $\{i,j,k\} = \{0,1,2\}$, and \begin{align*}\label{Err}\tag{4.11}
    \cE_{\bm b,\bm m,\bm r}({\bm X}) & = \sum_{\substack{ \forall i \in \{0,1,2\}, \; \ell_i \leqslant X_i \\ (\ref{cond cal E})} }  {\mu^2(\ell_{012}) \tau( \gcd(\pimp{(\ell_{012})},q) ) \over \tau(\pimp{(\ell_{012})})}\Theta(m_{12}\ell_0, m_{02}\ell_1, m_{01} \ell_2 ),
    \end{align*}where $\Theta$ is defined by (\ref{theta}) and the sum is over the conditions
          \[ \tag{4.12}\label{cond cal E} \left\{ \begin{tabular}{l}
              $\forall \{i,j,k\} = \{0,1,2\}, \; \gcd(\ell_i,b_j,b_k) =  \gcd(\ell_{012},m_{012}) = 1 $, \\
              $\forall i \in \{0,1,2\}, \; \ell_i \equiv r_i \; (\bmod \; q)$.
          \end{tabular}\right.\]
    
\begin{prop}\label{réécriture de Nbms(X)}

    Assume that $\displaystyle \prod_{p \mid q} \vartheta_p(\bm n) = 1$. If $\bm X \in [2, + \infty )^3$, $\bm r \in (\eZ/q \eZ)^3$, and $\bm b , \bm m~\in \eN^3 $ satisfy $($\ref{H1}$)$, $($\ref{H2}$)$, $($\ref{congr}$)$ and $($\ref{H3}$)$, the following estimate holds
    \begin{equation*}
      \cN_{\bm b,\bm m,\bm r} ({\bm X}) = {\tau(\gcd(\pimp{(m_{012})} , q)) \over \tau( \pimp{(m_{012})})} \left(   2 \cM_{\bm b,\bm m,\bm r}({\bm X})  + \cE_{\bm b,\bm m,\bm r} ({\bm X}) \right) +  O(\tau(q)^3),
    \end{equation*}
    where $ \cM_{\bm b,\bm m,\bm r}({\bm X})$ and   $\cE_{\bm b,\bm m,\bm r} ({\bm X})$ are respectively defined by (\ref{M}) and (\ref{Err}).
\end{prop}
\begin{proof}
    We write $\vartheta_p(C_{\bm m, \bm \ell})$ for $\vartheta_p(m_{12} \ell_0, m_{02}\ell_1, m_{01}\ell_2 )$. We follow the strategy of \cite[§5.5]{LRS}. It consists in writing the product
    \[ \prod_{p \nmid q} \vartheta_p(C_{\bm m, \bm \ell}) \] in terms of Hilbert symbols using \cite[chap. III. th.1]{Serre2}. We let $a := m_{12}\ell_0, b := m_{02}\ell_1$ and $c := m_{01}\ell_2$. If the integers $a,b,c$ are coprime in pairs and if the integer $abc$ has an odd prime divisor, then we write 
    \[ \prod_{p \nmid q} \vartheta_p(a,b,c)  =  \prod_{\substack{p \mid 2u \\ p \nmid q}}\left( {1 + (ac , bc)_p \over 2} \right) =  \prod_{\substack{p \mid u \\ p \nmid q}}\left( {1 + (ac , bc)_p \over 2} \right),\]where $u$ is the square-free integer given by the product of the odd prime divisors of $abc$. Note that the second equality comes from the fact that $q$ is even. By expanding the product, we obtain
        \[\prod_{p \nmid q} \vartheta_p(a,b,c) = {1 \over \tau\left( {u \over \gcd(u,q)} \right) } \prod_{\substack{p \mid u \\ p \nmid q}}\left( 1 + (ac , bc)_p  \right) =  {1 \over \tau\left( {u \over \gcd(u,q)} \right) }\left( 1 + \prod_{\substack{p \mid u \\ p \nmid q}} (ac,bc)_p + \Theta(a,b,c) \right), \]where $\Theta(a,b,c)$ is as in (\ref{theta}). We recall that if $p \nmid u$, then we have $(ac,bc)_p = 1$. When $u$ has a prime factor that does not divide $q$, the Hilbert product formula leads to 
        \[ \prod_{\substack{p \mid u \\ p \nmid q}} (ac,bc)_p  = \prod_{p \mid 2\gcd(q,u)}(ac , bc)_p. \]From (\ref{congr}), the condition $\ell_i \equiv r_i \Mod{q}$ and the assumption $\displaystyle \prod_{p \mid q} \vartheta_p(\bm n) = 1$ it follows that $\displaystyle \prod_{p \mid q} \vartheta_p(a,b,c) =~1$ and hence $(ac,bc)_p = 1$ for all $p \mid q$. Therefore, we have
           \[ \prod_{p \nmid q} \vartheta_p(a,b,c)  =  {1 \over \tau\left( {u \over \gcd(u,q)} \right) }  \left(  2  + \Theta(a,b,c) \right).\]
   The latter equality is satisfied when $u$ has a prime factor that does not divide $q$ so we have to deal with the triples $(a,b,c) = (m_{12}\ell_0,m_{02}\ell_1, m_{01}\ell_2) \in \eN^3$ such that $abc \mid q$ (in our case $abc$ is square-free), whose contribution in $\cM_{\bm b, \bm m, \bm r}$ is $O(\tau(q)^3)$. Since $u = \pimp{(m_{01}m_{02}m_{12}\ell_0\ell_1\ell_2)}$, it follows that
   \begin{equation*}
            \prod_{p \nmid q} \vartheta_p(C_{\bm m,\bm \ell} ) = { 2 + \Theta(m_{12}\ell_0, m_{02}\ell_1, m_{01} \ell_2 ) \over  \tau\left({\pimp{(m_{012}\ell_{012})} \over \gcd(\pimp{(m_{012}\ell_{012})}, q)} \right) } .
        \end{equation*}
        The conclusion comes from the multiplicativity of $\tau $ and from the fact that $m_{012}$ and $\ell_{012}$ are square-free and coprime.
\end{proof}

\paragraph{Step 3 : estimating $\cM_{\bm b,\bm m,\bm r}({\bm X})$.} We recall that
\begin{equation*}
        \cM_{\bm b,\bm m,\bm r}({\bm X}) := \sum_{\substack{ \forall i \in \{0,1,2\}, \; \ell_i \leqslant X_i \\ (\ref{cond cal E})} } {\mu^2(\ell_{012}) \tau( \gcd(\pimp{(\ell_{012})},q) ) \over \tau(\pimp{(\ell_{012})})}
    \end{equation*}
where $\bm r \in (\eZ/q \eZ)^3$, and where $\bm b$ and $\bm m$ satisfy the conditions (\ref{H1}), (\ref{H2}) and (\ref{H3}). 

\begin{prop}\label{calcul de MbmX}
    If $\bm b, \bm m \in \eN^3$ and $\bm X \in [2, + \infty )^3$ satisfy the conditions $($\ref{H1}$)$, $($\ref{H2}$)$, $($\ref{H3}$)$, and $($\ref{hyp odg Xi}$)$ for some $\eta > 0$ that does not depend on the $X_i$, if $q \in \eN$ satisfies $v_2(q) \geqslant 3$, $v_p(q) \geqslant 2$ for all odd prime $p$ dividing $q$, and if $q \leqslant (\log \min X_i )^{1/6}$, the estimate
    \[ \cM_{\bm b,\bm m,\bm r} ({\bm X}) = {\alpha({\bm b,\bm m,\bm r}) \beta_q({\bm b,\bm m})\over (2\pi)^{3/2} \varphi(q)^3}  \prod_{0 \leqslant i \leqslant 2} {X_i \over \sqrt{\log X_i}} + O_\eta \left( {X_0X_1X_2  \log_2( 3  \min X_i) \over \varphi(q)^3(\log \min X_i )^{5/2} }  \right) \]holds for all $\bm r \in (\eZ/ q \eZ)^3$, where the quantity $  \alpha( {\bm b,\bm m,\bm r} ) $ is defined by  
    \begin{align*}\label{expr alpha}\tag{4.13}
    \alpha( {\bm b,\bm m,\bm r} )  := & \mu^2(  \gcd(q, r_{012})) \mathds{1}(\gcd(q,r_{012},m_{012}) = 1) \!\!\!\!\!\! \prod_{\{i,j,k\} = \{0,1,2\}}\!\!\!\!\!\! \mathds{1}(\gcd(q,r_i,b_j,b_k)=1),
\end{align*} 
and where the quantity $\beta_q({\bm b,\bm m})$ is defined by 
    \begin{align*}\label{expr beta}\tag{4.14}
     \beta_q({\bm b,\bm m})& := \prod_{p > 2 } \left( 1 - {1 \over p} \right)^{3/2} \left( 1 + { \# \{ (i,j) : 0 \leqslant i < j \leqslant 2, \; p \nmid  m_{012}\gcd(b_i,b_j)q  \}  \over 2p } \right).
\end{align*}
\end{prop}

\begin{rmk}
The quantity $\alpha(\bm b, \bm m, \bm r)$ does not depend on the choice of a representative of $\bm r$ modulo~$q$.
\end{rmk}

The end of this step is dedicated to the proof of Proposition \ref{calcul de MbmX}. We start by replacing $\pimp{(\ell_{012})}$ by $\ell_{012}$ at the cost of a factor $2$ and extra notation. To this end, for every $\bm q = (q_0,q_1,q_2) \in \left\{ q, {q \over 2} \right\}^3$ and $\bm r' \in \displaystyle \prod_{i = 0}^2 \eZ / q_i \eZ$, we introduce
\begin{align*}
    T_{\bm r', \bm q}({\bm X}) := \sum_{\substack{\forall i \in \{0,1,2\}, \; \ell_i \leqslant X_i \\ (\text{\ref{cond T}})} }  {\mu^2(\ell_{012}) \tau(\gcd(\ell_{012},q)) \over  \tau(\ell_{012} ) },
\end{align*}where
    \[ \tag{4.15} \label{cond T} \left\{  \begin{tabular}{l}
     $\gcd(\ell_{012},m_{012}) = 1$, \\
     $ \forall \{i,j,k\}= \{0,1,2\}, \; \gcd(\ell_i,b_j,b_k)  = 1 $, \\
              $\forall i \in \{0,1,2\}, \; \ell_i \equiv r_i' \; (\bmod \; q_i)$\\
              $2 \nmid \ell_{012}$.
              \end{tabular} \right. \]Note that the dependence on $\bm r'$ comes from (\ref{cond T}). Looking separately at the cases $2 \nmid \ell_{012}$, $2 \mid \ell_0$, $2 \mid \ell_1$ and $2 \mid \ell_2$ enables us to write, since $2\mid q$,
              
              \begin{align*} \label{disjonction} \tag{4.16}\cM_{\bm b,\bm m,\bm r}({\bm X})  = & T_{\bm r, (q,q,q)}({\bm X}) \mathds{1}(2 \nmid r_0r_1r_2) \\
              & +  
               T_{\left( {r_0 \over 2},r_1,r_2\right),\left({q \over 2},q,q\right)} \left( {X_0 \over 2}, X_1, X_2\right)  \mathds{1}(2 \nmid m_{012}r_1r_2 \gcd(b_1,b_2), \; 2 \mid r_0)  \\
              & +  T_{\left( r_0,{r_1 \over 2},r_2\right),\left(q,{q\over 2},q\right)}\left( X_0, {X_1 \over 2}, X_2\right)  \mathds{1}(2 \nmid m_{012} r_0r_2 \gcd(b_0,b_2), \;  2 \mid r_1)\\
              & + T_{\left( r_0,r_1,{r_2 \over 2}\right),\left(q,q,{q\over 2}\right)}\left( X_0, X_1, {X_2 \over 2} \right)  \mathds{1}(2 \nmid m_{012}r_0r_1 \gcd(b_0,b_1), \; 2 \mid r_2) . \end{align*}
We focus on estimating $T_{\bm r} (\bm X) := T_{\bm r, (q,q,q)}({\bm X})$ since the other sums can be estimated by the same method.\\

The next step consists in transforming the congruence conditions into a sum over multiplicative characters, by using the orthogonality of characters. In order to do that, we need a coprimality condition between the $r_i$ and $q$. In order to force this coprimality condition, we factor out the integers $\gcd(r_i,q)$ thus giving rise to the quantity $\alpha(\bm b, \bm m, \bm r)$, defined in (\ref{expr alpha}), in the main term. This step is done in Lemma~\ref{reecriture Tn}. For $\{i,j,k\} = \{0,1,2\}$, we define
\[ \delta_i := \gcd(b_j,b_k)\gcd(q,r_{012}) m_{012}.\]For $\bm \chi = (\chi_0, \chi_1,\chi_2)$ any triple of Dirichlet characters with $\chi_i$ being a character modulo $ {q \over \gcd(q,r_i)}$, we introduce \begin{align*} \label{def h} \tag{4.17}
f_{\bm r}({\bm \ell}) = & { \mu^2(\ell_0\ell_1\ell_2)\tau\left(\gcd \left(\ell_{012}, { q \over \gcd (q,  r_{012})} \right) \right) \over  \tau(\ell_0\ell_1\ell_2) } \prod_{i=0}^2\mathds{1}(\gcd(\ell_i, \delta_i)=1), 
\end{align*} and 
 \[\label{Us}\tag{4.18} U_{\bm r,\bm \chi} ({\bm X}) := \sum_{\substack{\forall i \in \{0,1,2\}, \; \ell_i \leqslant X_i/\gcd(q,r_i)\\ \gcd\left( \ell_i , {q \over \gcd(q,\ell_i)} \right) \\ 2 \nmid \ell_{012} }} \chi_0(\ell_0)\chi_1(\ell_1)\chi_2(\ell_2) f_{\bm r}({ \bm \ell}). \]
 
\begin{lemme}\label{reecriture Tn} If $\displaystyle \gcd(q, r_{012},m_{012}) \prod_{\{i,j,k\} = \{0,1,2\}} \gcd(q,r_i,b_j,b_k) > 1$, then $T_{\bm r}({\bm X}) = 0$. Otherwise, we have
  \begin{align*}
    &T_{\bm r}({\bm X})= {\mu^2(  \gcd(q, r_{012})) \over \gcd(q,r_{012})^{-1} \varphi(q)^3   }  \sum_{\bm \chi} \overline{\chi_0}\left({r_0 \over \gcd(r_0,q)} \right) \overline{\chi_1}\left({r_1 \over \gcd(r_1,q)} \right) \overline{\chi_2}\left({r_2 \over \gcd(r_2,q)} \right)  U_{\bm r, \bm \chi} ({\bm X}), 
\end{align*}
where the sum is over the triples $\bm \chi = (\chi_0, \chi_1,\chi_2)$ with $\displaystyle \chi_i$ a Dirichlet character modulo $\displaystyle {q \over \gcd(q,r_i)}$.
\end{lemme}
\begin{proof}
  
Aiming to use orthogonality of Dirichlet characters, we make a change of variables in order to add a coprimality condition on the summation indexes. This relies on the equivalence

\[ \ell_i \equiv r_i \; (\bmod \; q)  \quad \Longleftrightarrow \quad {\ell_i \over \gcd(r_i, q)} \equiv {r_i \over \gcd(r_i, q)} \; \left( \bmod \; {q \over \gcd(r_i,q)}\right) \]and on the equality
\begin{align*} {\mu^2 (\ell_{012} ) \over  \tau(\ell_{012} ) } \tau(\gcd(\ell_{012},q))= &\; \mu^2( \gcd ( q, r_{012})) { \mu^2\left({\ell_{012} \over \gcd( q,  r_{012})}\right) \over  \tau\left({\ell_{012} \over  \gcd( q,  r_{012})}\right) }    \mathds{1}\left(\gcd\left(q,r_{012}, {\ell_{012} \over \gcd(q, r_{012})} \right) = 1\right).\end{align*} Indeed, we have $\gcd(\ell_{012},q) = \gcd(r_{012},q)$ and if $\gcd\left(q,r_{012}, {\ell_{012} \over \gcd(q, r_{012})} \right) > 1$, then $\ell_{012}$ is not square-free. We can now change the variable ${\ell_i \over \gcd(q,r_i)}$ into $\ell_i$ in order to recover the quantity $f_{\bm r}(\bm \ell)$. The new integers $\ell_i$ being invertible modulo ${q \over \gcd(q,r_i)}$, we can conclude by applying orthogonality of characters and noticing that we have
\begin{align*}
    & \gcd(\ell_0,b_1,b_2) = \gcd(\ell_1,b_0,b_2) = \gcd(\ell_2,b_0,b_1) = 1 \text{ and } \gcd(\ell_{012},m_{012}) \gcd(\ell_{012},q,r_{012}) = 1  \\ 
    & \Longleftrightarrow \forall i \in \{0,1,2\}, \; \gcd(\ell_i, \delta_i) = 1 \text{ where } \delta_i := \gcd(b_j,b_k)\gcd(q,r_{012}) m_{012}\text{ with } \{i,j,k\} = \{0,1,2\}.
\end{align*}Since $q$ is assumed to be square-full, it follows that
\[\mu^2(\gcd(q,r_{012}))\prod_{0\leqslant i \leqslant 2} {1 \over \varphi\left( {q \over \gcd(q,r_i)}\right)} = \mu^2(\gcd(q,r_{012})){\gcd(q,r_{012}) \over \varphi(q)^3}.\]
\end{proof}

\noindent We now estimate $ U_{\bm r,\bm \chi} ({\bm X})$. The following lemma is a direct consequence of \cite[lemma 1]{FI}. It enables us to neglect the contribution coming from the quantities $U_{\bm r, \bm \chi}$ when at least one of the $\chi_i$ is non-principal. 

\begin{lemme}\label{car non princ}
    Let $\bm \chi = (\chi_0,\chi_1,\chi_2)$ be a triple of Dirichlet characters with $\chi_i$ modulo ${q \over \gcd(q,r_i)}$ for each $i \in \{0,1,2\}$. Assume that there exists $i \in \{0,1,2\}$ such that $\chi_i$ is non-principal. If~$\bm X \in [2, + \infty )^3$, and if the inequalities $\max(b_0,b_1,b_2,m_{12},m_{02},m_{01})\leqslant (\log \min X_i)^C$ and $ q \leqslant (\log \min X_i)^{1/6}$ are satisfied, then we have the estimate
    \[ U_{\bm r,\bm \chi} (\bm X) \ll {X_0 X_1X_2 \over (\log  \min  X_i )^{5/2}} . \] 
\end{lemme}
\begin{proof}
    Without loss of generality, we can assume that $\chi_2$ is non-principal. We then write
    \begin{align*}
        U_{\bm r,\bm \chi} ({\bm X}) = & \sum_{\substack{ \ell_0 \leqslant X_0/ \gcd( q, r_0) \\ \gcd( \ell_0 , \delta_0) = 1 }} \chi_0 (\ell_0) {\mu^2(\ell_0) \tau \left( \gcd\left( \ell_0, {q \over \gcd(q, r_{012} )} \right) \right)   \over \tau(\ell_0) } \\
        &\times  \sum_{\substack{ \ell_1 \leqslant X_1/ \gcd( q, r_1) \\ \gcd( \ell_1 , \ell_0 \delta_1) = 1 }} \chi_1 (\ell_1) {\mu^2(\ell_1) \tau \left( \gcd\left( \ell_1, {q \over \gcd(q, r_{012} )} \right) \right)   \over \tau(\ell_1) } \\ & \times  \sum_{\substack{ \ell_2 \leqslant X_2/ \gcd( q, r_2) \\ \gcd( \ell_2 , \ell_0 \ell_1 \delta_2) = 1 }} \chi_2 (\ell_2)  {\mu^2(\ell_2) \tau \left( \gcd\left( \ell_2, {q \over \gcd(q, r_{012} )} \right) \right)   \over \tau(\ell_2) } \\
         \ll & \; \tau(q)^2 \sum_{\substack{ \ell_0 \leqslant X_0/ \gcd( q, r_0)  }} {1 \over \tau(\ell_0)}
         \sum_{\substack{ \ell_1 \leqslant X_1/ \gcd( q, r_1)  }} {1 \over \tau(\ell_1)} \\
         & \times \abs{ \sum_{\substack{ \ell_2 \leqslant X_2/ \gcd( q, r_2) \\ \gcd( \ell_2 , \ell_0 \ell_1 \delta_2) = 1 }} \chi_2 (\ell_2)  {\mu^2(\ell_2) \tau \left( \gcd\left( \ell_2, {q \over \gcd(q, r_{012})} \right) \right)   \over \tau(\ell_2) } }.
    \end{align*}Replacing ${1 \over \tau(\cdot)}$ by ${\tau(\gcd(\cdot , q)) \over \tau(\cdot)}$ in the proof of \cite[lemma 1]{FI} is legitimate, since it changes only a finite number of factors. It yields 
    \[ \sum_{\substack{ \ell_2 \leqslant X_2/ \gcd( q, r_2) \\ \gcd( \ell_2 , \ell_0 \ell_1 \delta_2) = 1 }} \chi_2 (\ell_2)  {\mu^2(\ell_2) \tau \left( \gcd\left( \ell_2, {q \over \gcd(q, r_{012} )} \right) \right)   \over \tau(\ell_2) } \ll_{C'} q \tau( \ell_0 \ell_1 \delta_2) {X_2 \over (\log \min X_i)^{C'}}\]for all $C' > 0$. Recall that we have $\delta_i := \gcd(b_j,b_k)\gcd(q,r_{012}) m_{012}$ for any $\{i,j,k\} = \{0,1,2\}$. We can therefore conclude by the assumptions $\max(b_0,b_1,b_2,m_{12},m_{01},m_{02})\leqslant(\log \min X_i)^C$, $q~\leqslant~(~\log \min X_i)^{1/6}$, and by choosing $C'$ large enough.
\end{proof}
 Thus, we can assume that every character $\chi_i$ is principal, and it remains to estimate the sum
 \[ U (\bm X ) := \sum_{\substack{\forall i \in \{0,1,2\}, \; \ell_i \leqslant X_i/\gcd(q,r_i)\\ \forall i \in \{0,1,2\}, \; \gcd\left(\ell_i, {q \over \gcd(q,r_i)}\right)= 1 \\ 2 \nmid \ell_{012}  }}  f_{\bm r}({ \bm \ell}). \]We recall that
 $f_{\bm r}$ is given by (\ref{def h}). Note that if $2 \nmid \ell_{012}$, the condition \[ \gcd\left(\ell_i, {q \over \gcd(q,r_i)} \right) \gcd(\ell_i, \delta_i) = 1\] is equivalent to $\gcd\left( \ell_i , \widetilde{\delta_i}\right)=1$ with $\widetilde{\delta_i} := \gcd(b_j,b_k)m_{012}q$ for $\{i,j,k\} = \{0,1,2\}$. Hence, we have
 \begin{equation*}\label{U}\tag{4.19}
     U(\bm X) = \sum_{\substack{\forall i \in \{0,1,2\}, \; \ell_i \leqslant X_i/\gcd(q,r_i)  }}  f({ \bm \ell}),
 \end{equation*}
with
\begin{equation*}\label{f}\tag{4.20}
    f(\bm \ell) := {\mu^2(\ell_0\ell_1\ell_2) \over  \tau(\ell_0\ell_1\ell_2) } \prod_{0 \leqslant i \leqslant 2} \mathds{1}\left( \gcd\left(\ell_i, 2\widetilde{\delta_i}\right) = 1\right).
\end{equation*}
    Fix $\bm z = \left(  {1 \over 2},  {1 \over 2},  {1 \over 2} \right)$ and recall that $\boldsymbol{\tau}_{\bm z}(n)$ is the $n$-th coefficient in the development of  $\zeta(s)^z$ ($\Re(s) > 1$) as a Dirichlet series. Let $g : \eN^3 \to \eC$ be the multiplicative function defined by $f = \boldsymbol{\tau}_{\bm z} \ast g$, meaning that
\[ \forall (\ell_0,\ell_1,\ell_2) \in \eN^3, \; f({\bm \ell}) = \!\!\!\!\!\!\! \sum_{\substack{{\bm k} \in \eN^3 \\ 
\forall i \in \{0,1,2\}, \; k_i \mid \ell_i}} \!\!\!\!\!\!\!\! \boldsymbol{\tau}_{\bm z} ({\bm k})g\left( {\ell_0 \over k_0}, {\ell_1 \over k_1}, {\ell_2 \over k_2} \right). \]Note that $g$ is well-defined since $\boldsymbol{\tau}_{\bm z}$ is invertible for the Dirichlet convolution and $D_g$ is holomorphic on $\{ \bm s \in \eC^3 : \underset{0 \leqslant i \leqslant 2}{\min} \sigma_i >1 \}$. In order to apply Lemma~\ref{lemme TN} we need the value of $D_g(\mathbf{1})$.\\

\begin{lemme}\label{série de dirichlet de h} The real number $D_g(\mathbf{1})$ is well-defined and we have $D_g(\mathbf{1}) = 2^{-3/2} \beta_q(\bm b, \bm m)$, where $\beta_q(\bm b, \bm m)$ is defined by $($\ref{expr beta}$)$.
\end{lemme}
\begin{proof}
Let $\bm s = (s_0, s_1, s_2) \in \eC^3 $ be such that $\Re (s_i) = \sigma_i > 1 $ for all $i \in \{0,1,2\}$. A direct calculation starting from (\ref{f}) leads to
\begin{align*} D_f (\bm s) & =  \prod_{\substack{p > 2 \\ p \nmid m_{012}q }} \left( 1 + {1 \over 2 } \left(  {\mathds{1}(p \nmid \gcd(b_1,b_2) ) \over p^{s_0} } + {\mathds{1}(p \nmid \gcd(b_0,b_2)  ) \over p^{s_1} } + {\mathds{1}(p \nmid \gcd(b_0,b_1)) \over p^{s_2} } \right) \right).\end{align*}Moreover, the Dirichlet series of $f$ can be written $D_f = D_{\boldsymbol{\tau}_{\bm z} }D_g$ with
\[ D_{\boldsymbol{\tau}_{\bm z}}(\bm s) = \zeta^{1/2} (s_0) \zeta^{1/2}(s_1) \zeta^{1/2} (s_2).\]We can conclude since the infinite product $\displaystyle\prod_p \left( 1 - {1 \over p} \right)^{3 \over 2}  \left( 1 + {3\over 2p}\right)$ is convergent.
\end{proof}

\begin{prop}\label{Estimation Tn}
If $\displaystyle \gcd(q, r_{012}, m_{012})  \!\!\!\!\! \prod_{ \{i,j,k \} =  \{0,1,2\}} \!\!\!\!\!\! \gcd(q,r_i,b_j,b_k)>1$, then $T_{\bm r}({\bm X}) = 0$. Otherwise, if the conditions $($\ref{H1}$)$, $($\ref{H2}$)$, $($\ref{H3}$)$, and $($\ref{hyp odg Xi}$)$ are satisfied for some $\eta > 0$ that does not depend on the $X_i$, if $q$ is such that $v_2(q) \geqslant 3$, $v_p(q) \geqslant 2$ for all odd prime $p$ dividing $q$, and if $q \leqslant (\log \min X_i )^{1/6}$, then we have the estimate
  \begin{align*}
    & T_{\bm r}({\bm X}) = {\mu^2(  \gcd( q, r_{012})) \over \varphi(q)^3}  { \beta_q({\bm b,\bm m}) \over (2\pi)^{3/2} }   \left\{ \prod_{0 \leqslant i \leqslant 2} {X_i \over  \sqrt{\log X_i}} \right\} + O_\eta \left( {X_0X_1X_2  \log_2( 3  \min X_i)  \over \varphi(q)^3 (\log \min X_i )^{5/2} }  \right).
\end{align*} 

\end{prop}
\begin{proof}
We apply Lemma \ref{lemme TN} with $\ell= 0$, $\bm z = \left( {1 \over 2} , {1 \over 2} , {1 \over 2} \right)$, $a_0(1/2) = \Gamma(1/2)^{-1} = \pi^{-1/2}$, $E_{\bm z, k} = 3/2$ and $f = \boldsymbol{\tau}_{\bm z} \ast g$ defined in (\ref{f}) in order to estimate the sum $ U (\bm X)$ which appears in (\ref{U}). Note that the equality $a_0(1/2) = \Gamma(1/2)^{-1}$ comes from the unicity of the development (\ref{def ai}), identifying the constant term with $\displaystyle \lim_{s \to 0^+} {(s \zeta(s+1))^{1/2} \over s+1} = 1$. To verify that Lemma \ref{lemme TN} can be applied, it suffices to provide an upper bound for the quantity \[\sum_{{\bm k} \in \eN^3} {\abs{g({\bm k})} \over k_0k_1k_2} (\log  \max (k_0,k_1,k_2))^{5/2} \]depending on $\bm b, \bm m, q$. Since $f = g \ast \boldsymbol{\tau}_{\bm z}$, we can write\begin{align*}
    D_g({\bm s}) & := \sum_{ \bm{k} \in \eN^3} {g({\bm k}) \over k_0^{s_0} k_1^{s_1} k_2^{s_2}} =   D_f({\bm s}) \prod_{0 \leqslant i \leqslant 2} \prod_{p } \left(1 - {1 \over p^{s_i} } \right)^{1/2} \quad \quad (\sigma_i > 1).
\end{align*}
We notice that we have 
\begin{align*} D_f (\bm s)  = & \prod_{p} \left( 1 + {1 \over 2}\left( {1 \over p^{s_0}} +  {1 \over p^{s_1}} +  {1 \over p^{s_2}}  \right) \right) \prod_{p \mid m_{012} q }  \left( 1 + {1 \over 2}\left( {1 \over p^{s_0}} +  {1 \over p^{s_1}} +  {1 \over p^{s_2}}  \right) \right)^{-1} \\
& \times \!\!\!\!\!\!\!\!\!\!\! \prod_{\substack{p \mid \gcd(b_0,b_1)\gcd(b_1,b_2) \gcd(b_0,b_2) \\
p \nmid m_{012}q}} \!\!\!\!\!\!\!\!\!\!\!\!\!\! { \left( 1 + {1 \over 2 } \left(  {\mathds{1}(p \nmid \gcd(b_1,b_2) ) \over p^{s_0} } + {\mathds{1}(p \nmid \gcd(b_0,b_2)  ) \over p^{s_1} } + {\mathds{1}(p \nmid \gcd(b_0,b_1)) \over p^{s_2} } \right) \right) \over  \left( 1 + {1 \over 2}\left( {1 \over p^{s_0}} +  {1 \over p^{s_1}} +  {1 \over p^{s_2}}  \right) \right) } \quad \quad (\sigma_i > 1).\end{align*}Therefore, we can write $g = g_1 \ast h$ where $g_1$ is the multiplicative function whose Dirichlet series is given by 
\begin{align*}D_{g_1} (\bm s) := & \prod_{p}    \left(1 - {1 \over p^{s_0}}  \right)^{1/2} \left(1 - {1 \over p^{s_1}}  \right)^{1/2}\left(1 - {1 \over p^{s_2}}  \right)^{1/2}  \left( 1 + {1 \over 2}\left( {1 \over p^{s_0}} +  {1 \over p^{s_1}} +  {1 \over p^{s_2}}  \right) \right), \quad \quad (\sigma_i > 1) \end{align*}and $h$ is the multiplicative function whose Dirichlet series is given by
\begin{align*} D_h(\bm s) := & \prod_{p \mid m_{012} q }  \left( 1 + {1 \over 2}\left( {1 \over p^{s_0}} +  {1 \over p^{s_1}} +  {1 \over p^{s_2}}  \right) \right)^{-1} \\ & \times \!\!\!\!\!\!\!\!\!\!\!\!\!\!\!\!\!\!\!\!\!\!\!\!  \prod_{\substack{p \mid \gcd(b_0,b_1)\gcd(b_1,b_2) \gcd(b_0,b_2) \\
p \nmid m_{012}q}} { \left( 1 + {1 \over 2 } \left(  {\mathds{1}(p \nmid \gcd(b_1,b_2) ) \over p^{s_0} } + {\mathds{1}(p \nmid \gcd(b_0,b_2)  ) \over p^{s_1} } + {\mathds{1}(p \nmid \gcd(b_0,b_1)) \over p^{s_2} } \right) \right) \over  \left( 1 + {1 \over 2}\left( {1 \over p^{s_0}} +  {1 \over p^{s_1}} +  {1 \over p^{s_2}}  \right) \right) } \quad \quad (\sigma_i > 1). \end{align*}Let $\Delta := m_{012}\gcd(b_0,b_1)\gcd(b_0,b_2) \gcd(b_1,b_2)q$. We follow \cite[exercise 202]{TW}. Note that for any $c > {1/2}$, we have the upper bounds
\[ \sum_{\bm k \in \eN^3} {\abs{g_1(\bm k)} \over k_0^{\sigma_0} k_1^{\sigma_1} k_2^{\sigma_2} } \ll_{c} 1 \quad \quad ( \sigma_i \geqslant c ) \]and
\[ \sum_{\bm k \in \eN^3} {\abs{h(\bm k)} \over k_0^{\sigma_0} k_1^{\sigma_1} k_2^{\sigma_2} } \ll_c \exp \left( \sum_{p \mid \Delta} {1 \over 2}\left( {1 \over p^{\sigma_0} } + {1 \over p^{\sigma_1} } + {1 \over p^{\sigma_2}} \right)  \right) \quad \quad ( \sigma_i \geqslant c ),\]from which we aim to deduce

\[ \label{hyp SD} \tag{4.21}\sum_{{\bm k} \in \eN^3} {\abs{g({\bm k})} \over k_0k_1k_2} (\log \max (k_0,k_1,k_2))^{5/2} \ll (\log_2 (5 \Delta) )^{3/2}. \]Indeed, let $p_\Delta$ be the $\omega(\Delta)$-th prime, and $\sigma_\Delta := 1-1/(3 + \log p_\Delta)$. We have the bound
\[ \sum_{p \mid \Delta} {1 \over p^{\sigma_\Delta}}\leqslant \sum_{p \leqslant p_\Delta} {1 + O(\log p / \log p_\Delta) \over p} \ll \log_2 p_\Delta + O(1) \ll \log_3(5\Delta) + O(1),\]which yields
\[ \sum_{{\bm k} \in \eN^3} {\abs{g({\bm k})} \over (k_0k_1k_2)^\sigma } \ll (\log_2(5\Delta))^{3/2} \quad \quad (\sigma \geqslant \sigma_\Delta)  . \]Inequality (\ref{hyp SD}) follows from the bound
\[ t^\kappa \leqslant \Gamma(\kappa+1) \mathrm{e}^t \quad \quad (t >0,\; \kappa \geqslant 0),\]for $t= \log(k_0k_1k_2)/(3 + \log p_\Delta)$ and $\kappa = 5/2$. We can thus apply Lemma \ref{lemme TN}, which provides for all $\bm X \in [2, + \infty )^3$ satisfying $($\ref{hyp odg Xi}$)$ for some $\eta > 0$ that does not depend on the $X_i$, the estimate
\begin{align*} U (\bm X) =  {  D_g(\mathbf{1})  \over \pi^{3/2}  \gcd( q,r_{012})  }  \! \left\{  \prod_{0 \leqslant i \leqslant 2} {X_i \over  \sqrt{\log {X_i \over \gcd(q,r_i)} }} \right\}\! +\! O_\eta \! \left( {X_0X_1X_2 (\log_2 (5 \max(b_i,m_{ij},q ) ))^{3\over 2}  \over \gcd (q, r_{012}) (\log \min X_i )^{5/2} }  \! \right). \end{align*}We now use the estimate
\[ \left(\log {X_i \over \gcd(q,r_i)} \right)^{-1/2} = (\log X_i )^{-1/2}\left( 1 - {\log \gcd (q,r_i) \over \log X_i } \right)^{-1/2} = (\log X_i )^{-1/2}\left( 1 + O \left( {\log_2 3\min X_i \over \log X_i} \right)\right), \]the bounds $\beta_q({\bm b,\bm m}) \leqslant \displaystyle \prod_{p >2} \left( 1 - {1 \over p} \right)^{3/2} \left( 1 + {3 \over 2p} \right)$ and $\max(b_0,b_1,b_2,m_{12},m_{02},m_{01}) \leqslant (\log \min X_i)^{C}$, and Lemma \ref{série de dirichlet de h}, to write
\begin{align*} U (\bm X) = & {  \beta_q({\bm b,\bm m}) \over (2\pi)^{3/2}  \gcd( q,r_{012})  }  \left\{ \prod_{0 \leqslant i \leqslant 2} {X_i \over  \sqrt{\log X_i}} \right\} +  O  \left( {X_0X_1X_2 \log_2 (3 \min X_i)\over \gcd (q, r_{012}) (\log \min X_i )^{5/2} }  \right). \end{align*} From Lemmas \ref{reecriture Tn} and \ref{car non princ}, it follows that
  \begin{align*}
    T_{\bm r}({\bm X}) = & {\mu^2(  \gcd( q, r_{012})) \over \varphi(q)^3}  { \beta_q({\bm b,\bm m}) \over (2\pi)^{3/2} }   \left\{ \prod_{0 \leqslant i \leqslant 2} {X_i \over  \sqrt{\log X_i}} \right\} \\
    & + O \left( {X_0X_1X_2 \log_2(3  \min X_i) \over  \varphi(q)^3  (\log \min X_i )^{5/2} }  \right).
\end{align*} This completes the proof of Proposition \ref{Estimation Tn}.
\end{proof}

\begin{rmk}\label{meme ODG}
    Under the same assumptions as in Proposition \ref{Estimation Tn}, and since $v_2(q) \geqslant 3$, we obtain \textit{mutatis mutandis}
   \begin{align*}  T_{\left( {r_0 \over 2},r_1,r_2\right),\left({q \over 2},q,q\right)} \left( {X_0 \over 2}, X_1, X_2\right) = & \; {\mu^2( \gcd( q, r_{012})) \over \varphi(q)^2 \varphi(q/2)}  { \beta_q({\bm b,\bm m}) \over (2\pi)^{3/2} } {X_0 \over 2 \sqrt{\log (X_0/2)}}  \left\{ \prod_{1 \leqslant i \leqslant 2} {X_i \over  \sqrt{\log X_i}} \right\} \\
   & + O \left( {X_0X_1X_2  \log_2( 3   \min X_i)  \over \varphi(q)^3 (\log \min X_i )^{5/2} }  \right)\\
   = & \;  {\mu^2(  \gcd( q, r_{012})) \over \varphi(q)^3 }  { \beta_q({\bm b,\bm m}) \over (2\pi)^{3/2} }  \left\{ \prod_{0 \leqslant i \leqslant 2} {X_i \over  \sqrt{\log X_i}} \right\} \\
   & + O \left( {X_0X_1X_2  \log_2( 3   \min X_i)  \over \varphi(q)^3 (\log \min X_i )^{5/2} }  \right)
   \end{align*}
  and the same estimate for $ T_{\left( r_0,{r_1 \over 2},r_2\right),\left(q,{q\over 2},q\right)}\left( X_0, {X_1 \over 2}, X_2\right)  $ and $T_{\left( r_0,r_1,{r_2 \over 2}\right),\left(q,q,{q\over 2}\right)}\left( X_0, X_1, {X_2 \over 2} \right)   $.
\end{rmk}

The conclusion of the proof of Proposition \ref{calcul de MbmX} thus follows from (\ref{disjonction}).

\paragraph{Step $4$: returning to $\cN_{\bm b,\bm m,\bm r}({\bm X})$.} It remains to estimate the quantity $\cE_{\bm b,\bm m,\bm r}({\bm X})$ defined by (\ref{Err}).

\begin{prop}\label{estimation reste}
   In the setting of Proposition \ref{réécriture de Nbms(X)}, for all $C' > 0$ we have the upper bound
    \[ \cE_{\bm b,\bm m,\bm r}({\bm X}) \ll_{C'} {X_0 X_1 X_2  \over (\log \min X_i)^{C'}}. \]
\end{prop}
\begin{proof} We estimate the quantity
\[ \sum_{\substack{ \forall i \in \{0,1,2\}, \; \ell_i \leqslant X_i \\ (\ref{cond cal E})} } \!\!\!\!\!\!\!\!\!\! {\mu^2(\ell_{012}) \tau( \gcd(\pimp{(\ell_{012})},q) ) \over \tau(\pimp{(\ell_{012})})} \!\!\!\!\!\!\sum_{\substack{{\bm u}, {\bm u'} \in \eN^3  \\
          (\ref{cond theta}) }}  \left( {m_{02}m_{01}\ell_1\ell_2 \over u_0}\right) \left( {m_{12}m_{01} \ell_0\ell_2 \over u_1}  \right) \left( {-m_{12}m_{02}\ell_0 \ell_1 \over u_2} \right). \]As in \cite[lemma~5.12]{LRS}, we can treat each prime factor of $q$ as $2$ and write $\ell_i = \gcd(\ell_i,q) d_id_i' = g_id_id_i'$, and $m_{ij} = \gcd(m_{ij},q) h_{ij} h_{ij}'$ to isolate a sum of the form 
          \[\sum_{\substack{\bm d , \bm d' \\ d_id_i' \leqslant X_i/g_i}} {\mu^2(d_{012}d_{012}') \over \tau(d_{012}d_{012}') } (-1)^{G_{\bm h}(\bm d)/4}\left( {g_1g_2 \over d_0h_{12}}\right) \left( { g_0g_2 \over d_1h_{02}}  \right) \left( {-g_0g_1 \over d_2h_{01}} \right) \left( {d_1'd_2' \over d_0h_{12}}\right) \left( { d_0'd_2' \over d_1h_{02}}  \right) \left( {-d_0'd_1' \over d_2h_{01}} \right),  \]
          where each $d_id_i'h_{jk}h_{jk}'$ is coprime to $q$ and where $(-1)^{G_{\bm h}(\bm d)/4}$ comes from the quadratic reciprocity law. Conditions (\ref{cond cal E}) and (\ref{congr}) impose in particular $d_i \equiv w_i \Mod{q/g_i}$ for some $\bm w \in \underset{0 \leqslant i \leqslant 2}{\prod}\eZ/ {q \over g_i}\eZ$. Note that choosing this congruence condition introduces a factor of size at most $q^3$. We now proceed to a slight adaptation of the calculation done in \cite[§5.5--5.7]{LRS} which relies on two results, respectively handling the contributions from the large conductors and from the small conductors. To deal with the large conductors, the authors of \cite{LRS} use a well-known result from Friedlander and Iwaniec \cite[lemma~2]{FI} providing the bound
\begin{equation*}\tag{4.22}\label{quadra sieve}
   \forall M,N \geqslant 2, \;  \sum_{\substack{1 \leqslant n \leqslant N \\ 1 \leqslant m \leqslant M}} a_nb_m \mu^2(2mn) \left({m \over n} \right) \ll (MN^{5/6} + NM^{5/6}) (\log(MN))^2
\end{equation*}for any complex sequences $(a_n)$ and $(b_m)$ satisfying $|a_n|, |b_m| \leqslant 1$ for all $n,m$, and where the implicit constant in the notation $\ll$ is absolute. In our case, we need to add congruence conditions on $n$ and~$m$, which can be done by replacing the complex number $a_n$ (\textit{resp.} $b_m$) by $a_n\mathds{1}(n \equiv r_1 \Mod{q / g_i})$ (\textit{resp.}~$b_m\mathds{1}(m \equiv r_2 \Mod{q/g_i})$). The upper bound for the contribution of the small conductors relies on \cite[lemma~5.17]{LRS} which is a slight modification of \cite[lemma 1]{FI}. This lemma states that the following bound holds
\begin{equation*}\tag{4.23}\label{smallcond}
   \forall X,Y \geqslant 2, \;  \sum_{\substack{Y \leqslant n \leqslant X \\ \gcd(n,D) = 1}} f(n) \chi(n){\mu^2(n) \over \tau(n)} \ll_{C'} {X \over (\log X)^{C'}} \tau(D) Q_0^2Q \max_{n \in \eN} |f(n)|
\end{equation*}for any $C'>0$, any periodic function $f : \eN \to \eC$ of period $Q_0$, for any $D \in \eN$ and non-principal Dirichlet character $\chi$ modulo $Q$ with $\gcd(Q,Q_0)=  1$. Similarly to \cite[lemmas 5.18 and 5.19]{LRS}, we employ this estimate to sums of the form

\[ \sum_{ \substack{ (\log \min X_i)^{c_1} <  d_2 \leqslant X_2 \\
d_2 \equiv w_2 \Mod{q/g_2} \\
\gcd(d_2,D) = 1}} f(d_2) \left( {d_2 \over d_0'd_1' h_{12}' h_{02}' } \right){\mu^2 (d_2) \over \tau(d_2)},  \]for some $c_1 > 0$. Note that in \cite[lemmas 5.18 and 5.19]{LRS}, this bound is used for $Q_0 \in \{1,8\}$ since in their case $f = 1$ or $f$ a product of $(-1)^{G_{\bm h}(\bm d)/4}$ and Legendre symbols $\left( {2 \over \cdot} \right)$. In our case, the primes dividing $Q_0 $ are divisors of $2g_0g_1$, since $f = 1$ or $f$ is a product of $(-1)^{G_{\bm h}(\bm d)/4}$ and Jacobi symbols $\left({g_0g_1\over \cdot}\right)$. To deal with the congruence condition modulo $q/g_2$, we use orthogonality of Dirichlet characters modulo $q/g_2$ ($d_2$ is invertible modulo $q$). It suffices then to replace $f$ by $\psi f$ with $\psi$ a Dirichlet character modulo $q/g_2$. Since $q$ is square-full, $\psi f$ is periodic of period coprime to $Q = d_0'd_1'h_{12}' h_{02}'$ so we can apply (\ref{smallcond}) and conclude as in \cite[§5.7]{LRS} by taking $C'$ large enough.
\end{proof}

\begin{cor}\label{estimation Nbms(X)} If $\bm b, \bm m \in \eN^3$ and $\bm X \in [2, + \infty )^3$ satisfy the conditions $($\ref{H1}$)$, $($\ref{H2}$)$, $($\ref{H3}$)$, and $($\ref{hyp odg Xi}$)$ for some $\eta > 0$ that does not depend on the $X_i$, if $q \in \eN$ is such that $v_2(q) \geqslant 3$, $v_p(q) \geqslant 2$ for all odd prime $p$ dividing $q$, and if~$q \leqslant (\log\min X_i )^{1/6}$, we have for all $\bm r \in (\eZ/q\eZ)^3$ the estimate
     \begin{align*}
      \cN_{\bm b,\bm m,\bm r}({\bm X}) = & {2 \alpha({\bm b,\bm m,\bm r}) \beta_q({\bm b,\bm m})\tau( \gcd(\pimp{(m_{012})},q)) \over (2 \pi)^{3/2} \varphi(q)^3 \tau( \pimp{(m_{012})})} \left\{ \prod_{i = 0}^2 {X_i \over \sqrt{\log(X_i)} } \right\}\\
      &+O_\eta\left( {X_0X_1X_2  \log_2( 3  \min X_i) \over \varphi(q)^3(\log \min X_i)^{5/2}}  \right)
    \end{align*}
    where $\alpha(\bm b, \bm m, \bm r)$ and $\beta_q(\bm b, \bm m)$ are respectively defined by $($\ref{expr alpha}$)$ and $($\ref{expr beta}$)$.
\end{cor}
\begin{proof}
    It suffices to combine Propositions \ref{réécriture de Nbms(X)}, \ref{calcul de MbmX} and \ref{estimation reste} with $C > 5/2$, for instance~$C~=~3$. 
\end{proof}

\paragraph{Step 5 : returning to $N'({\bm n},q,\bm B)$} We finally obtain the following estimate for $N'({\bm n},q,\bm B)$.

\begin{prop}\label{prop3.14}Let $q \in \eN$ be such that $v_2(q) \geqslant 3$ and $v_p(q) \geqslant 2$ for all odd prime $p$ dividing $q$. For $\bm n \in (\eZ\cap [0,q))^3$, and for $\bm b, \bm m \in \eN^3$ denote by $\gamma_{\bm n}(\bm b, \bm m)$ the quantity
\begin{align*} \label{gamma} \tag{4.24} 
\gamma_{\bm n}(\bm b, \bm m) :=  \# \left\{ \bm r \in (\eZ / q\eZ)^3 :  \begin{tabular}{c} $\forall \{i,j,k\} = \{0,1,2\}, \; b_i^2 m_{ij} m_{jk} r_i \equiv n_i \; (\bmod \; q)$ \\ $\forall \{i,j,k\} = \{0,1,2\}, \; \gcd(r_i,q,b_j,b_k)=1$ \\ $\forall p \mid q, \; p^2 \nmid r_{012}m_{012} $
\end{tabular}  \right\}.
\end{align*}
Let $B \geqslant 16$, and $\bm B = (B_0,B_1,B_2) \in \left[ {B \over (\log B)^2}, B \right]^3$. If $q \leqslant (\log B)^{1/7}$, $\max (v_2(n_0n_2),v_2(n_1n_2)) < v_2(q) - 2$ and $\max (v_p(n_0n_2),v_p(n_1n_2)) < v_p(q)$ for all odd prime $p$ dividing $q$, then $N'({\bm n},q,\bm B)$ is equal to 
    \begin{align*}
         & {2 \over \varphi(q)^3 (2 \pi)^{3/2} } \left( \prod_{0 \leqslant i \leqslant 2} {B_i \over (\log B_i)^{1/2}} \right)  \!\!\!\!\!\! \sum_{ \substack{ \bm b, \bm m\in \eN^3  \\ \gcd(b_0,b_1,b_2) = 1\\ \gcd(m_{ij},b_k) =  1 } } \!\!\!\!\!\!\!\!\!\!\!\!\! \mu^2(m_{012})  { \beta_q({\bm b,\bm m}) \gamma_{\bm n}(\bm b, \bm m)\tau(\gcd(\pimp{(m_{012})} , q)) \over b_{012}^2 m_{012}^2 \tau( \pimp{(m_{012})})}\\
        &+ O \left( {B_{012} (\log_2 B) q^3 \over \varphi(q)^3 (\log B)^{5/2}}   \right).
    \end{align*} 
\end{prop}
\begin{proof}
We start by making $\gamma_{\bm n}(\bm b , \bm m)$ appear as \[ \sum_{\substack{{\bm r} \in (\eZ/q\eZ)^3 \\ b_i^2m_{ij}m_{ik}r_i = n_i}}\alpha({\bm b,\bm m,\bm r}) =  \gamma_{\bm n}(\bm b, \bm m). \] Combining Lemma \ref{réécriture Nn(B)} with Corollary \ref{estimation Nbms(X)} for $C = 3$, it follows that $N'({\bm n},q,\bm B) $ is equal to
 \begin{align*}
        &{2 B_{012}\over \varphi(q)^3 (2 \pi)^{3/2}}\sum_{ \substack{ \bm b, \bm m \in \left[ 1, (\log B)^C \right]^3  \\ \gcd(b_0,b_1,b_2) = 1 \\ \gcd(m_{ij},b_k) =  1  } }  { \beta_q({\bm b,\bm m})\gamma_{\bm n} ({\bm b,\bm m}) \tau(\gcd(\pimp{(m_{012})} , q)) \over  \tau( \pimp{(m_{012})})}{\mu^2(m_{012}) \over m_{012}^2} \\
        & \times \prod_{\{i,j,k\} = \{0,1,2\}} {b_i^{-2} \over \sqrt{\log\left( {B \over b_i^2 m_{ij}m_{ik}}\right) } } 
         + O \left( E(\bm B)\right) + O \left( {B_{012}  \over (\log B)^3} \right),
    \end{align*} 
where
\[ E(\bm B) = {B_{012} \log_2 B \over \varphi(q)^3}\!\!\!\!\!\!\! \sum_{\substack{ {\bm b,\bm m} \in \eN^3 \\ \max (b_i, m_{ij}) \leqslant (\log B)^3 }}  { 1 \over (b_{012}m_{012})^2  \log\left( \min  {B_i \over b_i^2 m_{ij}m_{ik}}\right)^{5/2} }. \]Since the sum is over $\max(b_i, m_{ij}) \leqslant (\log B)^3$, we find that 
\begin{align*}
  &   {1 \over \sqrt{\log\left( {B_0 \over b_0^2 m_{01}m_{02}}\right) } } {1 \over \sqrt{\log\left( {B_1 \over b_1^2 m_{01}m_{12}}\right) } } {1 \over \sqrt{\log\left( {B_2 \over b_2^2 m_{12}m_{02}}\right) } } \\ & = 
     \left(\prod_{0 \leqslant i \leqslant 2} {1 \over (\log B_i)^{1/2}} \right) \left( 1 + O \left( {\log (b_{012}m_{012}) \over \log B}   \right) \right),
\end{align*}
so $N'({\bm n},q,\bm B)$ equals

\begin{align*}
         & {2 \over \varphi(q)^3 (2 \pi)^{3/2} } \left( \prod_{0 \leqslant i \leqslant 2} {B_i \over (\log B_i)^{1/2}}\right) \! \! \! \!\!\!\!\! \!\!\!\sum_{ \substack{ \bm b, \bm m\in \left[1,(\log B)^3\right]^3  \\ \gcd(b_0,b_1,b_2) = 1\\ \gcd(m_{ij},b_k) =  1 } } \! \! \!\!\!\!\!\!\!\!\!\!\!\!\!  \mu^2(m_{012})  { \beta_q({\bm b,\bm m})\gamma_{\bm n}({\bm b,\bm m})\tau(\gcd(\pimp{(m_{012})} , q)) \over  b_{012}^2 m_{012}^2 \tau( \pimp{(m_{012})})} \\ & + O\left( {B_{012}  \over \varphi(q)^3 (\log B)^{5/2}}   \sum_{ {\bm b,\bm m} \in \eN^3 }  {  \beta_q(\bm b, \bm m) \gamma_{\bm n}(\bm b, \bm m) \log(b_{012}m_{012})  \over (b_{012}m_{012})^2  }\right)  \\ & + O \left( {B^3 \log_2 B \over \varphi(q)^3 (\log B)^{5/2}}   \sum_{ {\bm b,\bm m} \in \eN^3 }  {\log(b_{012}m_{012})  \over (b_{012}m_{012})^2  }\right) + O\left( {B^3  \over (\log B)^3} \right). 
    \end{align*} To handle the error terms, we use the bounds
\[  \beta_q({\bm b,\bm m}) \leqslant \displaystyle \prod_{p >2} \left( 1 - {1 \over p} \right)^{3/2} \left( 1 + {3 \over 2p} \right)  \ll 1, \] $\gamma_{\bm n}(\bm b, \bm m) \leqslant q^3$ and the estimate
\[  \sum_{ {\bm b,\bm m} \in \eN^3 }  { \log(b_{012}m_{012})  \over (b_{012}m_{012})^2  } \ll \sum_{k \in \eN} {k^\varepsilon \over k^2}  \]where the right hand-side is convergent. Finally, we ignore the condition $\max (b_i,m_{ij}) \leqslant (\log B)^3$ thanks, again, to the previous bounds.
\end{proof}

\section{Computing the constant}\label{sec III5}

In this section, we fix a square-full integer $q$ divisible by $8$ and we fix $\bm n \in \eN^3$ such that \newline $\max (v_2(n_0n_2),v_2(n_1n_2)) < v_2(q) - 2$ and $\max (v_p(n_0n_2),v_p(n_1n_2)) < v_p(q)$ for all odd prime $p$ dividing~$q$. We define

\begin{align*} 
    &c : = {1 \over \varphi(q)^3}  \prod_{p \mid q} \vartheta_p(\bm n) \cdot \!\!\!\!\!\!\!\!\!  \sum_{ \substack{ \bm b \in \eN^3  \\ \gcd(b_0,b_1,b_2) = 1  } } \!\! {1 \over b_{012}^2}  \sum_{\substack{\bm m\in \eN^3 \\ \gcd(m_{ij},b_k) = 1 } }  {\mu^2(m_{012}) \over m_{012}^2} {\tau(\gcd(\pimp{(m_{012})} , q)) \beta_q({\bm b,\bm m})\gamma_{\bm n}({\bm b,\bm m}) \over  \tau( \pimp{(m_{012})})}
\end{align*}where $\gamma_{\bm n}(\bm b, \bm m)$ is the quantity defined by $($\ref{gamma}$)$ and $\beta_q(\bm b, \bm m)$ is defined by (\ref{expr beta}). We prove the following result.

\begin{lemme}\label{partie impaire}
    We have \begin{align*} 
    &c  = {1 \over \varphi(q)^3} \delta_2 \prod_{p > 2} \left(1 - {1 \over p} \right)^{3/2} \delta_p,
\end{align*}
where the factor $\delta_p$ for $p \geqslant 2$ is defined by $($\ref{facteurs delta}$)$.
\end{lemme}

 As in \cite[(5.22)]{LRS}, we separate the contribution associated with $2^{v_2(b_i)}$ and $2^{v_2(m_{ij})} $, from the contribution associated with $\pimp{\bm b}$ and $\pimp{\bm m}$. 

\paragraph{Separation of the contributions for $p = 2$ and for $p \neq 2$.} Firstly, we provide another expression for the factor $\gamma_{\bm n}(\bm b, \bm m)$ defined by $($\ref{gamma}$)$.

\begin{lemme}\label{réécriture gamma}
    Let $\bm b, \bm m \in \eN^3$. If there exists $\{i,j,k\} = \{0,1,2\}$ such that $\gcd(b_i^2 m_{ij}m_{ik}, q) \nmid n_i$, we have $\gamma_{\bm n}(\bm b, \bm m) = 0$. Otherwise we have the equality
    \begin{align*} \label{factorgamma} \tag{5.1}
\gamma_{\bm n}(\bm b, \bm m) = & \prod_{p \mid q} \left(  \prod_{\{i,j,k\} = \{0,1,2\}}  \!\!\!\!\!\!\!\!\!\!\!\! \gcd\left(b_i^2m_{ij}m_{ik},p^{v_p(q)}\right) \mathds{1}\left( v_p( b_i^2 m_{ij}m_{ik} ) = v_p(n_i) \right)  \right. \\
& + \!\!\!\!\!\!\!\!\! \prod_{\{i,j,k\} = \{0,1,2\}} \!\!\!\!\!\!\!\!\!\!\!\! \gcd\left(b_i^2m_{ij}m_{ik},p^{v_p(q)-\delta_{i,0} }\right)  \mathds{1}\left( \begin{tabular}{c}  $\gcd(p,b_1,b_2) =\gcd(p,m_{012})=1$ \\ $v_p( b_i^2 m_{ij}m_{ik} ) + \delta_{i,0} = v_p(n_i)$ \end{tabular} \right) \\
& + \!\!\!\!\!\!\!\!\! \prod_{\{i,j,k\} = \{0,1,2\}}\!\!\!\!\!\!\!\!\!\!\!\! \gcd\left(b_i^2m_{ij}m_{ik},p^{v_p(q)-\delta_{i,1} }\right) \mathds{1}\left( \begin{tabular}{c}  $\gcd(p,b_0,b_2) =\gcd(p,m_{012})=1$ \\ $v_p( b_i^2 m_{ij}m_{ik} ) + \delta_{i,1} = v_p(n_i)$ \end{tabular} \right) \\
& + \left. \!\!\!\!\!\!\!\!\! \prod_{\{i,j,k\} = \{0,1,2\}} \!\!\!\!\!\!\!\!\!\!\!\!  \gcd\left(b_i^2m_{ij}m_{ik},p^{v_p(q)-\delta_{i,2} }\right) \mathds{1}\left( \begin{tabular}{c}  $\gcd(p,b_0,b_1) =\gcd(p,m_{012})=1$ \\ $v_p( b_i^2 m_{ij}m_{ik} ) + \delta_{i,2} = v_p(n_i)$ \end{tabular} \right) \right),
\end{align*}where $\delta_{i,j}$ denotes the Kronecker symbol. In particular, the arithmetic function defined by 
\[ (\bm b, \bm m) \mapsto \mu^2(m_{012})\gamma_{\bm n}(\bm b, \bm m) \mathds{1}\left( \begin{tabular}{c}
    $\gcd( b_0,b_1,b_2) = 1$    \\
    $\forall \{i,j,k\} = \{0,1,2\}, \; \gcd( b_i, m_{jk} ) = 1$ \\
     $\forall \{i,j,k\} = \{0,1,2\}, \; \gcd(b_i^2 m_{ij}m_{ik}, q) \mid n_i$
\end{tabular} \right)\]is multiplicative.
\end{lemme}
\begin{proof}
    
We assume that for any $\{i,j,k\} = \{0,1,2\}$ we have $\gcd(b_i^2 m_{ij}m_{ik}, q) \mid n_i$. This ensures that the equation $b_i^2 m_{ij} m_{ik} r_i \equiv n_i \Mod{q}$ has solutions $r_i \in (\eZ /q\eZ)$. By considering the different values of $d_i :=\gcd(q,r_i)$, the quantity $\gamma_{\bm n}(\bm b,\bm m)$ can be written as

\[ \sum_{ \substack{ \bm d \in \eN^3 \\\forall i \in \{0,1,2\}, \; d_i \mid q}} \!\!\!\!\!\!\! \mu^2(d_{012})  \prod_{\{i,j,k\} = \{0,1,2\}} \!\!\!\!\!\!\gcd\left(b_i^2m_{ij}m_{ik},{q \over d_i}\right)   \mathds{1}\left( \begin{tabular}{c}  $\gcd(d_i,b_j,b_k) =\gcd(d_i,m_{012})=1$ \\ $\forall p \mid d_i, \; v_p( b_i^2 m_{ij}m_{ik} d_i) = v_p(n_i)$ \end{tabular} \right).\]Note that we used that $\max v_p(n_i) < v_p(q) $ for $p \mid q$, which ensures that the congruence $b_i^2 m_{ij} m_{ik} r_i \equiv n_i \Mod{q}$ implies $v_p(b_i^2 m_{ij} m_{ik} d_i) = v_p(n_i)$. We recognise the convolution of the two multiplicative functions
\[ \bm d \mapsto  \mu^2(d_{012})  \mathds{1}\left( \begin{tabular}{c}  $\gcd(d_i,b_j,b_k) =\gcd(d_i,m_{012})=1$ \\ $\forall p \mid d_i, \; v_p( b_i^2 m_{ij}m_{ik} d_i) = v_p(n_i)$ \end{tabular} \right)\]and
\[ \bm d \mapsto \prod_{\{i,j,k\} = \{0,1,2\}}\gcd\left(b_i^2m_{ij}m_{ik},d_i\right)   \]and equality (\ref{factorgamma}) follows. The multiplicativity comes from the fact that~$\bm n$ is fixed, so that for each $p$ dividing~$q$ the factor associated with $p$ in (\ref{factorgamma}) contains at most one non-zero term.

\end{proof}

We now focus on the sum 
\[  \sum_{ \substack{ \bm b \in \eN^3  \\ \gcd(b_0,b_1,b_2) = 1  } } \!\! {1 \over b_{012}^2}  \sum_{\substack{\bm m\in \eN^3 \\ \gcd(m_{ij},b_k) = 1 } }  {\mu^2(m_{012}) \over m_{012}^2} {\tau(\gcd(\pimp{(m_{012})} , q)) \beta_q({\bm b,\bm m})\gamma_{\bm n}({\bm b,\bm m}) \over  \tau( \pimp{(m_{012})})}.\]
\paragraph{Contribution for $p = 2$.} 

As $\beta_q(\bm b,\bm m)$ only depends on $\pimp{\bm b}$ and on $\pimp{\bm m}$, we only have to compute
\[ \sum_{\substack{ (\beta_0, \beta_1, \beta_2) \in \eN^3  \\ \min(\beta_0,\beta_1,\beta_2) = 0 } } {1 \over 4^{\beta_0 + \beta_1 + \beta_2}}  \sum_{\substack{(\mu_{12}, \mu_{02} , \mu_{01} )\in \eN^3 \\ \min(\mu_{ij},\beta_k) = 0 \\ 0 \leqslant \mu_{12}+ \mu_{02} +\mu_{01} \leqslant 1 } }  {\gamma_{\bm n}({(2^\beta),(2^\mu) })  \over 4^{ \mu_{12}+ \mu_{02} +\mu_{01} } }.\]By assumption, the triple $\bm n$ is fixed and it satisfies in particular $\max v_2(n_i) < v_2(q) - 2$. Thus, when we replace $\gamma_{\bm n}$ by its expression (\ref{factorgamma}), we can observe that only one term in the previous sum is non-zero. For instance, if $v_2(n_0)$ is even, and $v_2(n_1)$, $v_2(n_2)$ are odd, there exists only one couple $(\bm \beta, \bm \mu)$ with $0 \leqslant \mu_{12} + \mu_{02} + \mu_{01} \leqslant 1$ such that $v_2(n_i) = 2\beta_i + \mu_{ij} + \mu_{ik}$ for all $\{i,j,k\} = \{0,1,2\}$. For this couple, since $v_2(n_i) < v_2(q)-2$, we must have $\gcd( 2^{2\beta_i + \mu_{ij}+ \mu_{ik}}, 2^{v_2(q)}) = 2^{2\beta_i + \mu_{ij}+ \mu_{ik}} $, hence $\gamma_{\bm n} ((2^\beta), (2^\mu)) = 4^{\beta_0+\beta_1+\beta_2} \times 4^{\mu_{01} + \mu_{02} + \mu_{01}}$. The other cases are similar and it follows that
\[ \sum_{\substack{ (\beta_0, \beta_1, \beta_2) \in \eN^3  \\ \min(\beta_0,\beta_1,\beta_2) = 0 } } {1 \over 4^{\beta_0 + \beta_1 + \beta_2}}  \sum_{\substack{(\mu_{12}, \mu_{02} , \mu_{01} )\in \eN^3 \\ \min(\mu_{ij},\beta_k) = 0 \\ 0 \leqslant \mu_{12}+ \mu_{02} +\mu_{01} \leqslant 1 } }  {\gamma_{\bm n}({(2^\beta),(2^\mu) })  \over 4^{ \mu_{12}+ \mu_{02} +\mu_{01} } }=1.\]

\paragraph{Contribution for $p > 2$.}We introduce new variables $\bm b' $ and $\bm m'$ to denote $\pimp{\bm b}$ and $\pimp{\bm m}$, and we aim to write 
\[ c_{\text{odd}} := \sum_{ \substack{ \bm b' \in \eN^3  \\ \gcd(b_0',b_1',b_2') = 1 \\ 2 \nmid b_{012}'  } } {1 \over (b_{012}')^2}  \sum_{\substack{\bm m' \in \eN^3 \\ \gcd(m_{ij}',b_k') = 1 \\ 2 \nmid m_{012}' } }  {\mu^2(m_{012}') \over (m_{012}')^2} {\tau(\gcd(m_{012}' , q)) \beta_q({\bm b',\bm m'})\gamma_{\bm n}({\bm b',\bm m'}) \over  \tau( m_{012}')} \]as an infinite product over $p > 2$. In order to achieve this goal, we relate the summand to a multiplicative function by writing the quantity $\beta_q$ conveniently. As in \cite[(5.23)]{LRS}, we obtain the following result. 

\begin{lemme}\label{lemme reecriture beta} For $i,j,k$ such that $\{ i,j,k \} = \{0,1,2\}$, we define $\widetilde{\delta}_i(\bm b', \bm m') := m_{012}' \gcd(b_j',b_k') $. Let \[ \kappa_q := \prod_{p \nmid q} \left( 1 - {1 \over p} \right)^{3/2} \left( 1 + { 3 \over 2p } \right).\] We can write $\beta_q({\bm b',\bm m'})$ as 
\begin{align*}
    \beta_q({\bm b',\bm m'})
    &= h({\bm b', \bm m'})\kappa_q \prod_{\substack{p > 2 \\ p \mid q }}\left( 1 - {1 \over p}\right)^{3/2},
\end{align*} 
where
\[ h({\bm b', \bm m'}) := \prod_{\substack{p \nmid q \\p \mid \widetilde{\delta}_0(\bm b', \bm m')\widetilde{\delta}_1(\bm b', \bm m')\widetilde{\delta}_2(\bm b', \bm m')}} \left( 1 - { \# \{i \in \{0,1,2\} : p \mid \widetilde{\delta}_i(\bm b', \bm m') \} \over 2p+3} \right) \]defines a multiplicative function in $\bm b'$ and $ \bm m'$.

\end{lemme}

\begin{cor}We have 
    \[ c_{\text{odd}} =  \left(\prod_{\substack{p > 2 \\p \mid q}}\left(1 - {1 \over p} \right)^{3/2} \right) \left(\prod_{p \nmid q}\left(1 - {1 \over p} \right)^{3/2} {\left(1 + {1 \over p} + {1 \over p^2}\right)\left(1 + {1 \over 2p} + {1 \over p^2}  \right) \over \left(1 - {1 \over p^2}\right)^2} \right).\]
\end{cor}
\begin{proof}
    We combine  Lemma \ref{lemme reecriture beta} and Lemma \ref{réécriture gamma}. If $p \mid q$ is odd, we note that $h((p^\beta), (p^\mu)) = 1$. Similarly to the case $p=2$, we can compute $\gamma_{\bm n}((p^\beta) , (p^\mu))$ and it follows that the factor corresponding to $p$ in $c_\text{odd}$ is $\left(1 - {1 \over p} \right)^{3/2}$. If~$p \nmid q$, we recover exactly the same factor as in  \cite[lemma 5.24]{LRS}, thus concluding the calculation of $c_\text{odd}$ and~$c$.
\end{proof}
 
This completes the proof of Theorem \ref{problème de comptage}.

\section{Proof of Theorem \ref{THA}.}\label{sec III6}

We begin by stating, without proof, a result from \cite{DLS}. Let $\omega : [1 , + \infty)^3 \to~\eC$ be a $C^1$ map and for all $q \in \eN$ and $\bm a \in (\eZ / q \eZ)^3$ let $\rho(\bm a, q)$ be a complex number. Let $g : \eN^3 \to \eC$ be an arithmetic function. For $B \in \eR_{\geqslant 1}$, we introduce the quantity

\begin{equation*}\label{E}\tag{6.1} E(B, q) := \sup_{ \bm B \in [1, B]^3} \max_{\substack{ \bm a  \in (\eZ / q \eZ)^3 \\ \gcd (\bm a, q) = 1}} \abs{  \sum_{\substack{ \bm m \in  \prod_{i= 0}^2 \eN \cap [1, B_i] \\ \bm m \equiv \bm a \Mod{q} }  } g(\bm m) - \rho(\bm a,q)  \int_{\prod_{i=0}^2 [1, B_i] } \omega( \bm t) \mathrm{d} \bm t }.\end{equation*}This quantity is a way to measure the equidistribution of $g$ in arithmetic progressions.\\

\begin{defin}\label{def W et eps}
    Let $z \geqslant 2$ and $p \in \cP \mapsto m_p(z) \in [1 , + \infty) $. We define 
    \[ W_z := \prod_{p \leqslant z} p^{m_p(z)}, \; \varepsilon(z) :=  \sum_{p \leqslant z} {1 \over p^{m_p(z)/2 } } \quad  \text{ and } \quad \widetilde{\varepsilon}(z) := \sum_{p \leqslant z} {1 \over p^{m_p(z) + 1} }. \]
\end{defin}Following \cite{DLS}, we choose 
\begin{equation*}\label{m_p}\tag{6.2}
    m_p(z) :=  \lceil z \rceil + 2 \mathds{1}_{p=2},
\end{equation*}so that $m_2(z) \geqslant 3$ and $\varepsilon(z)$, $\widetilde{\varepsilon}(z) \to 0$ when $z $ goes to $ + \infty$. In what follows, $\mathcal{K}$ is a subset of $[-1,1]^{n+1}$ satisfying 
\begin{equation*}\label{boite}\tag{6.3} \mathcal{K} := \prod_{j = 0}^n [u_j , u_j'] \subset [-1,1]^{n+1} \text{ with } \max_{0 \leqslant j \leqslant n} \abs{u_j - u_j'} \leqslant 1.\end{equation*} The main result we are using is the following

\begin{tho}[Destagnol--Lyczak--Sofos]\cite[th. 2.4]{DLS}\label{theoreme kevin}
    Let $n \in \eN$ and $(F_0,F_1,F_2)$ be three homogeneous polynomials as in Definition \ref{polynomes de birch} where $d \geqslant 1$ denotes their common degree. Let $\boldsymbol{\epsilon} \in \{-1,1\}^3$. Let $z \geqslant 2$ and $W_z$, $\widetilde{\varepsilon} (z)$ as in Definition \ref{def W et eps}. Let $g : \eN^3 \to \eC$ be an arithmetic function. Then there exist constants $c, \delta > 0$ such that for all $B \geqslant 1$ we have the estimate
    \begin{align*} {1 \over B^{n+1}} \sum_{\substack{ \bm k \in \eZ^{n+1} \cap B\mathcal{K} \\ \min_j \epsilon_j F_j(\bm k) > 0  }} g(\epsilon_0F_0(\bm k), \epsilon_1 F_1(\bm k), \epsilon_2 F_2(\bm k)) = & \sum_{ \bm r \in (\eZ / W_z)^{n+1} } { \rho( (\epsilon_i F_i(\bm r)), W_z) \over W_z^{n-2} } \\
    & \times \!\!\!\!\!\!\!\!\!\! \underset{ \substack{ \bm \nu \in \mathcal{K} \\ \min_j \epsilon_j F_j(\bm \nu) > B^{-d} } }{\int} \!\!\!\!\!\!\!\!\!\omega( B^d \epsilon_0 F_0(\bm \nu) , B^d \epsilon_1 F_1(\bm \nu) ,  B^d \epsilon_2 F_2(\bm \nu) ) \mathrm{d}\bm \nu \\
    & + O \left( {\abs{\abs{g}}_1 \over B^{3d} }( B^{-\delta} + \widetilde{\varepsilon}(z) + z^{-c} ) + {E( bB^d , W_z ) W_z^3 \over B^{3d} }  \right),
    \end{align*}
    where the implicit constant depends at most on the $F_i$ and  
    \[ \abs{\abs{g}}_1 = \sum_{\substack{\bm t \in \eN^3 \cap [1 , bB^d]^3 } } \abs{g (\bm t)}, \]and $b=2 \max_i ( \max \abs{F_i(\mathcal{K})})$.
\end{tho}

The following result will also be useful. This is a corollary of theorem \ref{theoreme kevin} for the indicator function of the set $\mathcal{A}$.

\begin{lemme}\label{lemme utile crible}
    Keep the setting of theorem \ref{theoreme kevin}. Let $\mathcal{A} \subset \left[ 1, bB^{d}\right]^3$. Then we have
    \[ { \#\{ \bm k \in \eZ^{n+1} \cap B\mathcal{K} : \min_j \epsilon_j F_j(\bm k) > 0 \text{ and } (\epsilon_0F_0(\bm k) , \epsilon_1 F_1(\bm k) , \epsilon_2 F_2(\bm k) ) \in \mathcal{A}  \} \over B^{n+1} } \ll {\# \mathcal{A} \over B^{3d} },\]where the implicit constant can depend on the  $F_i$ but does not depend on $\mathcal{A}$.
\end{lemme}

\begin{rmk}\label{partition boîte}
    We cannot take $\mathcal{K} = [-1,1]^{n+1}$ in theorem \ref{theoreme kevin}. We divide $\mathcal{K}$ into $2^{n+1}$ boxes $\mathcal{K}_i$ of the same volume, that is 
    \[ [-1,1]^{n+1} = \bigsqcup_{i = 1}^{2^{n+1}} \mathcal{K}_i \]where $\mathrm{vol}(\mathcal{K}_i) = \mathrm{vol}(\mathcal{K}_j) = 1$ whenever $i \neq j$. It is clear that $\mathcal{K}_i$ satisfies $($\ref{boite}$)$ so that we apply theorem~\ref{theoreme kevin} for $B\mathcal{K}_i$. 
\end{rmk}

Using theorem \ref{theoreme kevin} and Lemma \ref{lemme utile crible}, we now focus on estimating the quantity 
\[ N_{\mathrm{loc}} (\pi_{\bm F},B^{n+1})  = \# \left\{ \bm x \in \p^n(\eQ) : \max |x_i| \leqslant B, \; \vartheta_\eQ (F_0({\bm x}),  F_1({\bm x}), F_2({\bm x})) = 1\right\}. \]To this end, for $\boldsymbol{\epsilon} \in \{-1,1\}^3$ we let 
\begin{equation*}\label{N'}\tag{6.4} N_{\boldsymbol{\epsilon}}'(\pi_{\bm F},B)  :=  \# \left\{ {\bm k} \in \eZ^{n+1} \! : \begin{tabular}{c} $\max \abs{k_i} \leqslant B$,\\ $\min_j \epsilon_j F_j(\bm k) > 0$, \\ $\vartheta_\eQ (\epsilon_0 F_0({\bm k}), \epsilon_1 F_1({\bm k}), \epsilon_2 F_2({\bm k}))=1$ \end{tabular} \right\}.\end{equation*}

The point is that, provided a precise estimate of each $N_{\boldsymbol{\epsilon}}'(\pi_{\bm F},B)$, a Möbius inversion yields 
\begin{align*} N_{\mathrm{loc}} (\pi_{\bm F},B^{n+1}) & = {1 \over 2} \sum_{\boldsymbol{\epsilon} \in \{-1,1\}^3} \sum_{\substack{\bm k \in \eZ^{n+1} \\
\max_i |k_i| \leqslant B \\ \min_j \epsilon_j F_j(\bm k) > 0 }} \sum_{\ell \mid \gcd(k_0, \dots ,k_n)} \!\!\!\!\! \mu(\ell) \vartheta_\eQ (\epsilon_0 F_0({\bm k}), \epsilon_1 F_1({\bm k}), \epsilon_2 F_2({\bm k})) \\
& = {1 \over 2} \sum_{\boldsymbol{\epsilon} \in \{-1,1\}^3}  \sum_{\ell \leqslant B^{n+1} }  \mu(\ell) N_{\boldsymbol{\epsilon}}'(\pi_{\bm F},B/\ell), \end{align*}where we used the fact that $F_0,F_1,F_2$ are homogeneous of the same degree $d$. Now, the trivial bound 
\[ N_{\boldsymbol{\epsilon}}'(\pi_{\bm F},B/\ell) \leqslant \left({B \over \ell} \right)^{n+1} \]yields
\[ \left| \sum_{\ell > \log(B) }  \mu(\ell) N_{\boldsymbol{\epsilon}}'(\pi_{\bm F},B/\ell) \right| \leqslant  B^{n + 1}  \sum_{\ell > \log(B) } {1 \over \ell^{n+1}} \leqslant {B^{n + 1}\over (\log B)^n},\]where $n \geqslant 2$ since $(F_0,F_1,F_2)$ is a Birch system. Hence, 
\[ \label{inv mobius} \tag{6.5} N_{\mathrm{loc}} (\pi_{\bm F},B^{n+1}) = {1 \over 2} \sum_{\boldsymbol{\epsilon} \in \{-1,1\}^3}  \sum_{\ell \leqslant \log(B) }  \mu(\ell) N_{\boldsymbol{\epsilon}}'(\pi_{\bm F},B/\ell) + O\left( {B^{n+1} \over (\log B)^2}\right),\]and we are led to estimate each quantity $N_{\bm \epsilon}'(\pi_{\bm F},B) $. For $z > 0$, we introduce the set 
\[ \mathcal{E}_{\bm \epsilon,z}(\pi_{\bm F}) := \left\{ {\bm k} \in \eZ^{n+1} :\; \begin{tabular}{l}
            $\forall p \in (2, z],\; \max (v_p( F_0({\bm k})F_2({\bm k})),v_p(F_1({\bm k})  F_2({\bm k}))  ) < m_p(z)$; \\
        $\max (v_2( F_0({\bm k})F_2({\bm k})),v_2(F_1({\bm k})  F_2({\bm k}))  )  \leqslant m_2(z) -3$; \\
        $\min_j \epsilon_j F_j(\bm k) > 0, \; \vartheta_\eQ (\epsilon_0 F_0({\bm k}), \epsilon_1 F_1({\bm k}), \epsilon_2 F_2({\bm k}))=1;$\\ $\forall p > z, \; p \nmid \gcd(F_0(\bm k), F_1(\bm k), F_2(\bm k))$
    \end{tabular}   \right\},  \]and we start by showing that $N_{\bm \epsilon}'(\pi_{\bm F},B)$ is well approximated by the quantity 
\begin{equation*}\label{N tilde} \tag{6.6}\widetilde{N}_{\bm \epsilon, z}(\pi_{\bm F},B) := \# \left(\mathcal{E}_{\bm \epsilon, z}(\pi_{\bm F}) \cap [-B,B]^{n+1}\right)\end{equation*} if $z$ is large enough and $m_p(z) $ is defined in \ref{def W et eps}. For such $z$, following Remark \ref{partition boîte}, we introduce

 \begin{equation*}\label{N tilde K} \tag{6.7} \widetilde{N}_{\bm \epsilon, z}(\pi_{\bm F},B, \mathcal{K})  := \# \left(\mathcal{E}_{\bm \epsilon, z}(\pi_{\bm F}) \cap B\mathcal{K}\right),\end{equation*} where $\mathcal{K}$ satisfies (\ref{boite}). In order to bound the difference between $ \widetilde{N}_{\bm \epsilon, z}(\pi_{\bm F},B, \mathcal{K})$ and $N_{\bm \epsilon}'(\pi_{\bm F},B) $, we use the large sieve as stated in \cite[lemma 5.1]{Wilson2}.  

\begin{lemme}\label{grand crible}
For any subset $Y_p \subseteq (\eZ/ p^2 \eZ)^3$ and for all $x_0,x_1,x_2 \geqslant 1$, $L \leqslant (\min x_i)^{1/4}$, we have
\[ \# \{ \bm s \in \eZ^3 : \abs{s_i} \leqslant x_i,  \forall p \leqslant L, \;  \bm s \; (\bmod \; p^2) \notin Y_p\} \ll {x_0x_1x_2 \over F( L )}, \]where

\[ F(L) = \sum_{a \leqslant L} \mu^2(a) \prod_{p \mid a} { \# Y_p \over p^6 - \# Y_p}.\]
\end{lemme}

 \begin{lemme}\label{applic crible}
     Let $x_0,x_1,x_2$, $c_0,c_1,c_2$ be natural numbers satisfying $c_{012} \leqslant (\min x_i)^{1/6}$. We have
     \[ \# \left\{ \bm t \in \eZ^3 : \abs{t_i} \leqslant x_i,  c_i \mid t_i, (t_0t_2,t_1t_2)_\eQ = 1  \right\} \ll {1\over c_{012}  }{x_{012} \over (\log \min x_i)^{3/2}},\]where the implicit constant is absolute.
 \end{lemme}
 \begin{proof}
      Let 
     \[ A(\bm c, \bm x) := \# \left\{ \bm t \in \eZ^3 : \abs{t_i} \leqslant x_i,  c_i \mid t_i, (t_0t_2,t_1t_2)_\eQ = 1  \right\}.  \]We can write for each $i \in \{0,1,2\}$, $t_i = c_i s_i$ with integers $s_0,s_1,s_2$. Thus, 
     \[ A(\bm c, \bm x) \ll \# \left\{ \bm s \in \eZ^3 : \abs{s_i} \leqslant {x_i \over c_i},  (c_0c_2s_0s_2,c_1c_2s_1s_2)_\eQ = 1  \right\}.\]For $p \nmid 2c_{012}$, let $Y_p \subseteq (\eZ / p^2 \eZ)^3$ be the classes of $\bm s$ modulo $p^2$ such that $(s_0s_2,s_1s_2)_p = -1$. If $v_p(s_{012}) = 0$ then $(s_0s_2,s_1s_2)_p = 1$ so there is no contribution to $Y_p$. We now discuss the case $v_p(s_0s_1s_2) = 1$. If $v_p(s_0s_1) = 1$ and $v_p(s_2)= 0$, we have 
    \[ (s_0s_2,s_1s_2)_p \in \left\{ \left( s_0s_2\over p \right) , \left( s_1s_2\over p \right) \right\} \]so this case contributes $2p^3 {(p - 1)^2 \over 2}$~to~$\# Y_p$. If $v_p(s_2) = 1$ and $v_p(s_0s_1) = 0$, we have 
    \[ (s_0s_2,s_1s_2)_p =  (-1)^{{p-1 \over 2}} \left( s_0s_2/p\over p \right) \left( s_1s_2/p\over p \right).   \]There are respectively $p$, $\varphi(p^2)$, and $p {p - 1 \over 2}$ choices for $s_2,s_1$ and $s_0$ so this case contributes $p^3 {(p-1)^2 \over 2}$ to $\# Y_p$. It remains to discuss the case $v_p(s_0s_1s_2) \geqslant 2$. We have either $v_p(s_i) \geqslant 2$, for some $i \in \{0,1,2\}$ or $\min(v_p(s_i), v_p(s_j))\geqslant 1$, for $i \neq j$ in $\{0,1,2\}$. Each of these cases contributes at most $p^4$ to $ \# Y_p$, so we deduce that 
    \[ \# Y_p \geqslant {3 \over 2} p^3(p-1)^2  +O(p^4) \geqslant  {3 \over 2}p^5 -3p^4   .\]Hence, the function $F$ in lemma \ref{grand crible} satisfies
    \[ F(L) \gg \sum_{\substack{a \leqslant L \\ \gcd(2c_{012},a) = 1}} \mu^2(a) \prod_{p \mid a} {{3 \over 2} \left(p -2 \right) \over p^2 - {3 \over 2} \left(p -2 \right) } = \sum_{\substack{a \leqslant L \\ \gcd(2c_{012},a) = 1}} {\mu^2(a)u(a) \over a} \left( {3 \over 2}\right)^{\omega(a)},\]with $u$ the multiplicative function defined by
    \[ u(a) = \prod_{p \mid a} \left( {p^2-2p \over p^2 - {3 p\over 2} + 3}  \right) = \prod_{p \mid a} \left( 1-  {{ p \over 2} + 3 \over p^2 - {3 p\over 2} +3}  \right). \]To estimate the right-hand side, we write uniquely each integer $k$ as $k = ab^2$ with $a$ square-free, which yields
    \[ \sum_{\substack{k \leqslant L \\ \gcd(k, 2c_{012} )}} {1 \over k} \left({3 \over 2} \right)^{\omega(k)}\!\!\!\!u(k) = \sum_{\substack{ab^2 \leqslant L \\ \gcd(ab^2, 2c_{012} ) = 1} } {\mu^2(a) \over ab^2} \left( {3 \over 2} \right)^{\omega(ab^2)}\!\!\!\!\! u(ab^2) \ll\sum_{\substack{a \leqslant L \\ \gcd(2c_{012},a) = 1}} {\mu^2(a)\over a} \left( {3 \over 2}\right)^{\omega(a)}\!\!\!\!\! u(a).  \]Then, we apply the Selberg--Delange method to estimate the sum 
    \[  \sum_{\substack{k \leqslant L \\ \gcd(k, 2c_{012} )}}  \left({3 \over 2} \right)^{\omega(k)}\!\!\!\!\! u(k).\]Let $G(s)$ denote the product 
    \[ G(s) := \prod_{p} \left( 1 - {1 \over p^s} \right)^{3/2} \left( \sum_{\nu \geqslant 0} \left({3 \over 2}\right)^\nu { u(p^\nu) \over p^{\nu s}} \right) \]which is convergent for $\sigma > 1/2$, since $u(p) = 1 + O(1/p)$. Similarly to \cite[exercise 203, p181]{TW}, it follows that
    
    \[ \sum_{\substack{k \leqslant L \\ \gcd(k, 2c_{012} )}}  \left({3 \over 2} \right)^{\omega(k)}\!\!\!\!\! u(k) = L (\log L)^{1/2} \left( {G(1) \over \Gamma(3/2)}\prod_{p\mid 2c_{012}} \left( \sum_{\nu \geqslant 0} \left( {3 \over 2}\right)^\nu {u(p^\nu) \over p^\nu} \right)^{-1} + O\left({1 \over \log L} \right)\right).\]Again, since $u(p)= 1 + O(1/p)$, we have 
    \[ \prod_{p\mid 2c_{012}} \left( \sum_{\nu \geqslant 0} \left( {3 \over 2}\right)^\nu {u(p^\nu) \over p^\nu} \right)^{-1} \gg \prod_{p \mid 2c_{012} } \left( 1 + {3 \over 2p} \right)^{-1} \geqslant \left({\varphi(2c_{012} ) \over 2c_{012}}\right)^{3/2}.  \]Hence
    
    \[ \sum_{\substack{k \leqslant L \\ \gcd(k, 2c_{012} )}}  \left({3 \over 2} \right)^{\omega(k)}\!\!\!\!\! u(k)  \gg \left({\varphi(2c_{012} ) \over 2c_{012}}\right)^{3/2}L (\log L)^{1/2}.\]In particular, it follows from a partial summation that 
    \[ F(L) \gg \left({\varphi(2c_{012} ) \over 2c_{012}}\right)^{3/2} \int_{2}^L {\log( v)^{1/2} \over v} \mathrm{d}v \gg \left({\varphi(2c_{012} ) \over 2c_{012}} \log L\right)^{3/2}. \]We can therefore conclude with lemma \ref{grand crible} and $L = (\min x_i)^{5/24 } \leqslant \min (x_i/c_i)^{1/4}$.

 \end{proof}

 We remove here the large valuations and impose the coprimality condition to recover the assumptions from Theorem \ref{problème de comptage}.

\begin{cor}\label{lien N' et N tilde}
    Let $z > 2 $ be a real number. The quantity $\varepsilon(z)$ is as in Definition \ref{def W et eps}, and goes to $0$ when the parameter $z$ goes to $+\infty$. If $z \ll \log_2 B$, we have the estimate
    \begin{align*} N_{\bm \epsilon}'(\pi_{\bm F},B) & = \sum_{i=1}^{2^{n+1}} \widetilde{N}_{\bm \epsilon, z}(\pi_{\bm F},B,\mathcal{K}_i) + O\left( \left(\varepsilon(z) + {1 \over z \log z}\right) {B^{n+1} \over  (\log B)^{3/2} } + {B^{n+1/2} \over (\log B)^{1/2}}  \right),
\end{align*}where \[ [-1,1]^{n+1} = \bigsqcup_{i = 1}^{2^{n+1}} \mathcal{K}_i\] is the partition defined in Remark \ref{partition boîte}.
\end{cor}
\begin{proof}
It suffices to work with $ \mathcal{K} = \mathcal{K}_i$ fixed. Recall that $d$ is the common degree of the polynomials $F_i$ and $b$ denotes the real number~$2 \max_i ( \max \abs{F_i(\mathcal{K})})$. We start by dealing with the large valuations. The condition $\max (v_p(n_0n_2),v_p(n_1n_2))\geqslant  m_p(z)$ implies that $\max(v_p(n_0),v_p(n_1),v_p(n_2)) \geqslant m_p(z)/2$ so we can apply lemma \ref{lemme utile crible} to the set 
    \[ \mathcal{A} := \left\{ \bm t \in \eZ^3 \cap B\mathcal{K}: (t_0t_2,t_1t_2)_\eQ = 1, \; \exists p \leqslant z \text{ such that } \max v_p(t_i) \geqslant m_p(z)/2 \right\}.\]The result follows from Lemma \ref{applic crible} applied for instance with $c_0 = p^{m_p(z)/2}$, $c_1=c_2= 1$ for which we have 
    \[ c_{012} = p^{m_p(z)/2} \leqslant W_z \leqslant \mathrm{e}^{m_p(z) \left(z + {z \over 2 \log z}\right)} \ll B^{1/6},\]since $z \ll \log_2(B)$. Thus,
    \[ {\# \mathcal{A} \over B^{3d}} \ll \sum_{ p \leqslant z} {1 \over p^{m_p(z)/2} } { 1\over (\log B)^{3/2}} \ll { \varepsilon(z)  \over  (\log B)^{3/2}}. \]To deal with the relative coprimality of the $F_i(\bm k)$ with respect to $p > z$, we follow \cite[lemma 4.1]{DLS} when $R =3$ which provides a bound of size 
    \[ \ll {1 \over z \log z} {B^{n+1} \over (\log B)^{3/2}} + {B^{n+1/2 } \over (\log B)^{1/2}}.\]
\end{proof}

We now focus on estimating $ \widetilde{N}_{\bm \epsilon, z}(\pi_{\bm F},B,\mathcal{K}_i)$. Recall that we denote by $Q_{\bm t}$ the conic of equation 
$ t_0 y_0^2 + t_1y_1^2 - t_2y_2^2 = 0 $, with $\bm t \in \eN^3$. Thanks to theorem \ref{theoreme kevin} and Corollary \ref{lien N' et N tilde}, it suffices to find a non-trivial bound for the quantity $E(B,W_z)$ as in (\ref{E}). If there exists $i \in \{0,1,2\}$ such that $B_i < B^2/(\log B)^2$, then we handle $E(B,W_z)$ using the trivial bound. In the remaining cases, we have $\bm B \in \left[ B/ (\log B)^2 , B\right]$, so that we are led to provide an asymptotic formula for the quantity
\[  N({\bm n},W_z,\bm B)\! := \! \# \left\{  {\bm t} \in \eN^3 \! : \begin{tabular}{c} $\forall i \in \{0,1,2\}, \; t_i \leqslant B_i$; \\ $ \gcd(t_0,t_1,t_2) = 1,  \; \bm t \equiv \bm n \Mod{W_z} $ \\ $ Q_{\bm t}$ has a solution in $\eQ\smallsetminus \{(0,0,0)\}$,\end{tabular} \right\},\]where $W_z := \displaystyle \prod_{p \leqslant z} p^{m_p(z)}$ with $m_p(z)$ given by (\ref{m_p}), and $\bm n \in [0,W_z)^3$ is such that $\gcd(n_0,n_1,n_2,W_z) =~1$, $\max (v_2(n_0n_2),v_2(n_1n_2)) \leqslant m_2(z) - 3$ and $\max (v_p(n_0n_2),v_p(n_1n_2)) < m_p(z)$ for all $2  < p \leqslant z$. A satisfying estimate for $N(\bm n, W_z,B)$ is provided by Theorem \ref{problème de comptage}.

\begin{lemme}\label{sieg walf} Let $z > 2$ be a real number and let $W_z$ be as in Definition \ref{def W et eps} with $m_p(z)$ given by $($\ref{m_p}$)$. Let $\bm n \in [0, W_z)^3$ be such that $\gcd(n_0,n_1,n_2,W_z) =1$, $\max v_2(n_i) \leqslant m_2(z) - 3$ and $\max v_p(n_i) < m_p(z)$ for all $2  < p \leqslant z$. Let $B \geqslant 16$, and $\bm B = (B_0,B_1,B_2) \in \left[ {B \over (\log B)^2}, B \right]^3$. If~$z \leqslant \sqrt{ {1 \over 14} \log_2 B}$, then, when $B$ goes to $+\infty$, we have 
\[ N({\bm n} ,W_z,\bm B) = \rho({\bm n} , W_z) \prod_{0 \leqslant i \leqslant 2} {B_i \over (\log B_i)^{1/2}} + O \left( {(\log_3B)^{3/2} B_{012} \over W_z^3 \log_2 B (\log B)^{3/2}}\right), \]where
\[ \rho({\bm n} , W_z) : = {2 \over \pi^{3 /2}} {1 \over W_z^3}  \prod_{p \leqslant z} \left( 1 - {1 \over p} \right)^{-3/2}\vartheta_p(\bm n).
 \label{rho} \tag{6.8}  \]
\end{lemme}

\begin{proof}
     The assumption $z \leqslant \sqrt{ {1 \over 14} \log_2 B}$ ensures that $W_z\leqslant(\log B)^{1 / 7}$ so that we can apply Theorem~\ref{problème de comptage}, leading to
    \begin{align*}
     N({\bm n},W_z,\bm B) = &  {2 \over \pi^{3/2}\varphi(W_z)^3}\left( \prod_{p \leqslant z}\left( 1 - {1 \over p}\right)^{3/2} \vartheta_p(\bm n) \right)\\ & \times \left( \prod_{p > z}\left(1 - {1 \over p} \right)^{3/2}  {\left(1 + {1 \over p} + {1 \over p^2}\right)\left(1 + {1 \over 2p} + {1 \over p^2}  \right) \over \left(1 - {1 \over p^2}\right)^2} \right) \prod_{0\leqslant i \leqslant 2} {B_i \over (\log B_i)^{1/2}} \\
     & + O \left( {B_{012} \log_2(B) W_z^3\over \varphi(W_z)^3 (\log B)^{5/2}} \right).
     \end{align*}Now, we use the estimate
     \begin{align*} \prod_{p > z}\left(1 - {1 \over p} \right)^{3/2}  {\left(1 + {1 \over p} + {1 \over p^2}\right)\left(1 + {1 \over 2p} + {1 \over p^2}  \right) \over \left(1 - {1 \over p^2}\right)^2} & = \prod_{p > z}\left(1 + O\left( {1 \over p^2} \right)\right)= 1 + O\left( {1 \over z} \right) \end{align*}in order to write
     \begin{align*}
     N({\bm n},W_z,\bm B) = & \;  \rho(\bm n, W_z) \prod_{0\leqslant i \leqslant 2} {B_i \over (\log B_i)^{1/2}}   + O \left( {B_{012} \log_2(B) W_z^3 \over \varphi(W_z)^3 (\log B)^{5/2}} \right) \\ & + O\left( {1 \over z \varphi(W_z)^3} {B_{012} \over (\log B)^{3/2}} \prod_{p \leqslant z} \left( 1 - {1 \over p}\right)^{3/2} \right) .
     \end{align*}
     The last error term is admissible since 
     \[ \prod_{p \leqslant z} \left( 1 - {1 \over p}\right)^{3/2} = {\varphi(W_z)^{3/2} \over W_z^{3/2} } \quad \text{ and } \quad  {1 \over \varphi(W_z)} \ll {\log z \over W_z}, \]which leads us to the estimate
     \[ {1 \over z \varphi(W_z)^3} \prod_{p \leqslant z} \left( 1 - {1 \over p}\right)^{3/2} \prod_{0\leqslant i \leqslant 2} {B_i \over (\log B_i)^{1/2}} \ll {(\log z)^{3/2} \over z W_z^3}  {B_{012}  \over (\log B)^{3/2}} \ll  {(\log_3B)^{3/2} B_{012} \over W_z^3 \log_2 B (\log B)^{3/2}}. \]
\end{proof}

Let $g$ be the arithmetic function defined by 
 \[ \tag{6.9}\label{g} g(\bm t) := \mathds{1}\left( \begin{tabular}{l}
           $\forall p \in (2, z], \; \max (v_p(t_0t_2),v_p(t_1t_2)) < m_p(z)$ ; \\
        $\max (v_2(t_0t_2),v_2(t_1t_2)) \leqslant m_2(z) - 3$ ; \\
        $\forall p > z, \; p \nmid \gcd(t_0,t_1,t_2), \text{ and } \vartheta_\eQ(\bm t)=1$
    \end{tabular}   \right).\]Note that for $\bm n \in (\eZ/W_z \eZ)^3$ such that $\gcd(\bm n, W_z) = 1$ and $\bm t \equiv \bm n \Mod{W_z}$, we have 
    \[ g(\bm t) = \mathds{1}\left( \begin{tabular}{l}
           $\forall p \in (2, z], \; \max (v_p(t_0t_2),v_p(t_1t_2)) < m_p(z)$ ; \\
        $\max (v_2(t_0t_2),v_2(t_1t_2)) \leqslant m_2(z) - 3$ ; \\
        $ \gcd(t_0,t_1,t_2) = 1, \text{ and } \vartheta_\eQ(\bm t)=1$
    \end{tabular}   \right). \]Lemma \ref{sieg walf} provides a non-trivial upper bound for the quantity $E(bB^d, W_z)$ defined by (\ref{E}), relatively to $g$, with \[ \tag{6.10} \label{func omega} \omega(\bm y) := \prod_{0 \leqslant i \leqslant 2} {1 \over \sqrt{\log(y_i)}} \quad \quad (y_i  > 1)\]and $\rho$ as in (\ref{rho}). Until the end of this section, we take
    \[ z = \sqrt{ {1 \over 14} \log_2 B}. \tag{6.11} \label{val z}\]

\begin{cor}\label{bound E}
 If $z$ is as in (\ref{val z}) and if $\rho$, $g$ and $\omega$ are respectively defined by $($\ref{rho}$)$, $($\ref{g}$)$ and $($\ref{func omega}$)$, then, when $B$ goes to $+\infty$, we have
    \[ {  E(bB^d, W_z) W_z^3 \over B^{3d}} \ll {(\log_3B)^{3/2} \over  \log_2 B (\log B)^{3/2}}.\]
\end{cor}
We are now in a position to apply theorem \ref{theoreme kevin}.

\begin{cor}\label{estim Ntilde}
   Recall that the quantity $\widetilde{N}_{\bm \epsilon,z}(\pi_{\bm F},B,\mathcal{K})$ is defined by $($\ref{N tilde K}$)$. For $B \geqslant 2$ and $z$ as in (\ref{val z}), we have
    \begin{align*}
       { \widetilde{N}_{\bm \epsilon, z}(\pi_{\bm F},B,\mathcal{K}) \over B^{n+1}} = & \left(\int_{\substack{{\bm \nu} \in \mathcal{K} \\ \min \epsilon_j F_j({\bm \nu}) > B^{-d}}} \omega(B^d \epsilon_0 F_0({\bm \nu}), B^d \epsilon_1 F_1({\bm \nu}), B^d \epsilon_2  F_2({\bm \nu})) \mathrm{d} \bm \nu \right) \\
       & \times \left( \sum_{{\bm r}\in (\eZ/W_z \eZ)^{n+1} } {\rho( (\epsilon_0 F_0({\bm r}),\epsilon_1 F_1({\bm r}), \epsilon_2 F_2({\bm r}) ), W_z) \over W_z^{n-2}} \right) \\
        & + O\left( {1 \over (\log B)^{3/2}} \left( B^{-\delta} + \widetilde{\varepsilon}(z) + z^{-c} \right) \right) +  O\left( {(\log_3B)^{3/2} \over  \log_2 B (\log B)^{3/2}}\right),
    \end{align*}  
    where $\rho$ and $\omega$ are respectively defined by $($\ref{rho}$)$ and $($\ref{func omega}$)$.
\end{cor}
\begin{proof}
    This is a direct application of theorem \ref{theoreme kevin} to $g$ defined by (\ref{g}), combined with Lemma \ref{sieg walf} and Corollary \ref{bound E}. Note that \cite[th. 1.1]{LRS} ensures that 
    \[ ||g||_1 \ll {B^{3d} \over (\log B)^{3/2}}, \]so that the error term coming from theorem \ref{theoreme kevin} is admissible.
\end{proof}

\paragraph{End of the proof of Theorem \ref{THA}.} We take $z$ as in (\ref{val z}), $\omega$ as in~(\ref{func omega}) and $\rho$ as in (\ref{rho}). It remains to calculate the main term obtained in Corollary~\ref{estim Ntilde}, to sum over~$\mathcal{K}_i$ defined in Remark \ref{partition boîte} and over $\bm \epsilon \in \{-1,1\}^3$.

\begin{lemme}\label{int}
    In the setting of theorem \ref{theoreme kevin}, for $\omega $ as in $($\ref{func omega}$)$, we have the estimate \begin{align*}\underset{\substack{{\bm \nu} \in \mathcal{K} \\ \min \epsilon_j F_j({\bm \nu}) > B^{-d}}}{\int} \!\!\!\!\!\!\!\!\!\!\!\!\omega(B^d \epsilon_0 F_0({\bm \nu}), B^d\epsilon_1  F_1({\bm \nu}), B^d \epsilon_2 F_2({\bm \nu})) \mathrm{d} \bm \nu  = { {\rm vol}\left({\bm \nu} \in \mathcal{K} : \min \epsilon_j F_j(\bm \nu) \geqslant 0\right) \over (d \log B)^{3/2}} \left( 1 + O\left( {1 \over \log B} \right) \right),\end{align*} when $B$ goes to $+ \infty$.
\end{lemme}
\begin{proof}
    We write 
    \begin{align*} \underset{\substack{{\bm \nu} \in \mathcal{K} \\ \min \epsilon_j F_j({\bm \nu}) > B^{-d}}}{\int} \!\!\!\!\!\!\!\!\!\!\!\omega(B^d \epsilon_0 F_0({\bm \nu}), B^d \epsilon_1 F_1({\bm \nu}), B^d \epsilon_2 F_2({\bm \nu})) \mathrm{d} \bm \nu & = {1 \over (d\log B)^{3/2}}\!\!\!\!\!\!\!\!\!\!\underset{\substack{{\bm \nu} \in \mathcal{K} \\ \min \epsilon_j F_j({\bm \nu}) > B^{-d}}}{\int} \!\!\!\!\!\! \prod_{0 \leqslant i \leqslant 2}{1 \over \sqrt{ 1+ {\log(\epsilon_i F_i({\bm \nu})) \over d \log(B)}} } \mathrm{d} \bm \nu.\end{align*}Note that for $B$ large enough and $\bm \nu \in \mathcal{K}$ such that $\min_j F_j(\bm \nu)>B^{-d}$ we have $\log(\epsilon_i F_i(\bm \nu)) / \log B \in (0,1)$. Thus, we have
    \[ {1 \over \sqrt{ 1+ {\log (\epsilon_i F_i({\bm \nu})) \over \log (B^d)} }} = 1 + O \left( {1 \over \log B} \right) \]so that
    \begin{align*} \underset{\substack{{\bm \nu} \in \mathcal{K} \\ \min \epsilon_j F_j({\bm \nu}) > B^{-d}}}{\int} \!\!\!\!\!\!\!\!\!\!\!\!\!\!\!\!\! \omega(B^d \epsilon_0 F_0({\bm \nu}), B^d \epsilon_1 F_1({\bm \nu}), B^d \epsilon_2 F_2({\bm \nu})) \mathrm{d} \bm \nu = & {\text{vol}\left({\bm \nu} \in \mathcal{K} :  \min \epsilon_j F_j(\bm \nu) > B^{-d} \right) \over (d\log B)^{3/2}  }\!\!\left(1  + O \left( {1 \over \log B}\right)\right) . \end{align*}From \cite[lemma 1.19]{Browning}, it follows that 
    \[ \text{vol}\left({\bm \nu} \in \mathcal{K} : \min \epsilon_j F_j(\bm \nu) > B^{-d} \right) = \text{vol}\left({\bm \nu} \in \mathcal{K} : \min \epsilon_j F_j(\bm \nu)\geqslant  0 \right) \left( 1 + O \left( {1 \over \log B} \right) \right),\]which concludes the proof of Lemma \ref{int}.
\end{proof}

Finally, we have the following result coming from the definition of $\rho$ in Lemma \ref{sieg walf}.

\begin{lemme}\label{sum t}We have 
\[ \lim_{z \to + \infty } { 1 \over W_z^{n-2}} \underset{{\bm r}\in (\eZ/W_z \eZ)^{n+1} }{\sum} \rho( (\epsilon_0 F_0({\bm r}), \epsilon_1 F_1({\bm r}),  \epsilon_2 F_2({\bm r}) ), W_z)=
      {2 \over \pi^{3 /2}}\prod_p \left( 1 - {1 \over p} \right)^{-3/2} c_p, 
    \]where $c_p$ is defined by $($\ref{cp1}$)$.
\end{lemme}

Combining Lemma \ref{lien N' et N tilde} with Corollary \ref{estim Ntilde} and Lemma \ref{int}, we obtain that there exist $\delta, c > 0$ such that for $z$ as in (\ref{val z}),
\begin{align*} N_{\bm \epsilon}'(\pi_{\bm F} , B) =&  {\text{vol} \left({\bm \nu} \in [-1,1]^{n+1} : \min \epsilon_j F_j(\bm \nu)\geqslant  0 \right) \over d^{3/2}}\left( \sum_{{\bm r}\in (\eZ/W_z \eZ)^{n+1} } {\rho( (\epsilon_0 F_0({\bm r}),\epsilon_1 F_1({\bm r}), \epsilon_2 F_2({\bm r}) ), W_z) \over W_z^{n-2}} \right) \\
& \times {B^{n+1} \over (\log B)^{3/2} } + O \left( {B^{n+1} \over (\log B)^{3/2} } \left(B^{-\delta}+ z^{-c} + \widetilde{\varepsilon}(z) + {1 \over z \log z} + \varepsilon(z) + {(\log_3 B)^{3/2} \over \log_2(B) } \right)\right) .
\end{align*}

Replacing each $N_{\bm \epsilon}'(\pi_{\bm F} , B)$ by this expression in (\ref{inv mobius}) and then applying Lemma \ref{sum t} with $z = \sqrt{ {1 \over 14} \log_2 B} $ yields 
\[ N_{\text{loc}}(\pi_{\bm F} , B^{n+1}) \underset{B \to + \infty}{\sim} {2 \over (d\pi)^{3/2} } c_\infty \left( \prod_p \left( 1 - {1 \over p} \right)^{-3/2} c_p \right) {B^{n+1} \over (\log B)^{3/2} }, \]where we used the partition defined in Remark \ref{partition boîte} and where $c_\infty$ is the constant defined by (\ref{cinfty1}), for which we have
\[  \sum_{ \substack{\boldsymbol{\epsilon} \in \{-1,1 \}^3 \\ \boldsymbol{\epsilon} \neq \pm (1,1,-1)}}\text{vol}\left({\bm \nu} \in [-1,1]^{n+1} : \min_j \epsilon_j F_j(\bm \nu)\geqslant  0 \right) = c_\infty.\]
 
\textbf{Acknowledgement} : I am grateful to Jean-Louis Colliot-Th\'el\`ene and Cyril Demarche for their help with Proposition~\ref{brauer subordonné}. I thank Daniel Loughran for answering some questions about his paper \cite{LS}. I am also grateful to R\'egis de la Bret\`eche for his suggestion concerning Lemma \ref{lemme TN} and for his valuable writing advice. I finally thank Kevin Destagnol for his guidance throughout the writing of this paper. I am grateful to the anonymous reviewer for helpful remarks on an early version of this paper.

\end{document}